\UseRawInputEncoding
\documentclass[12pt,reqno]{amsart}
\usepackage{fullpage}

\newtheorem{theorem}{Theorem}[section]
\newtheorem{lemma}[theorem]{Lemma}

\newtheorem{cor}[theorem]{Corollary}
\newtheorem{conj}[theorem]{Conjecture}

\newtheorem{prob}[theorem]{Problem}
\usepackage{graphicx}
\usepackage{color}
\usepackage[dvipsnames]{xcolor}
\usepackage{subfigure}
\usepackage{amssymb}
\usepackage{amsmath,mathrsfs}
\usepackage{colonequals}
\usepackage{hyperref}
\theoremstyle{definition}
\newtheorem{definition}[theorem]{Definition}

\newtheorem{assumption}[theorem]{Assumption}
\newtheorem{remark}[theorem]{Remark}

\renewcommand{\subset}{\subseteq}
\renewcommand{\supset}{\supseteq}
\renewcommand{\epsilon}{\varepsilon}

\newcommand{\abs}[1]{\left|#1\right|}                   % Absolute value notation
\newcommand{\vnorm}[1]{\left\|#1\right\|}    % norm notation
\newcommand{\vnormf}[1]{\|#1\|}                         % norm notation, forced to be small
\newcommand{\vnormt}[1]{\left\|#1\right\|}    % norm notation
\newcommand{\E}{\mathbb{E}}
\renewcommand{\L}{L}
\newcommand{\mL}{\mathcal{L}}

\newcommand{\R}{\mathbb{R}}

\newcommand{\embolden}[1]{\textbf {#1}}
\newcommand{\redA}{\Sigma}
\newcommand{\redb}{\partial^{*}}
\newcommand{\redS}{\partial^{*}\Sigma}
\newcommand{\sdimn}{n}
\newcommand{\adimn}{n+1}
\newcommand{\scon}{\lambda}
\newcommand{\pcon}{\delta}

\begin{document}

\title{The Structure of Gaussian Minimal Bubbles}

\author{Steven Heilman}
\address{Department of Mathematics, University of Southern California, Los Angeles, CA 90089-2532}
\email{stevenmheilman@gmail.com}
\date{\today}
\thanks{Supported by NSF Grant DMS 1829383.}
%60E15, 60G15, 53A10, 58E30
\subjclass[2010]{60E15, 53A10, 60G15, 58E30}
\keywords{Gaussian, bubble, minimal surface, calculus of variations}

\begin{abstract}
It is shown that $m$ disjoint sets with fixed Gaussian volumes that partition $\mathbb{R}^{n}$ with minimum Gaussian surface area must be $(m-1)$-dimensional.  This follows from a second variation argument using infinitesimal translations.  The special case $m=3$ proves the Double Bubble problem for the Gaussian measure, with an extra technical assumption.  That is, when $m=3$, the three minimal sets are adjacent $120$ degree sectors.  The technical assumption is that the triple junction points of the minimizing sets have polynomial volume growth.  Assuming again the technical assumption, we prove the $m=4$ Triple Bubble Conjecture for the Gaussian measure.  Our methods combine the Colding-Minicozzi theory of Gaussian minimal surfaces with some arguments used in the Hutchings-Morgan-Ritor\'{e}-Ros proof of the Euclidean Double Bubble Conjecture.
\end{abstract}
\maketitle
%
%
% arxiv subjects: math.DG, math.PR, cs.CC?
%
%  MSC:    60E15, 60G15, 53A10, 58E30
%
%  keywords: Gaussian, bubble, minimal surface, calculus of variations
%
% 35? pages, 4 figures

\section{Introduction}\label{secintro}

%\snote{for future; try doing stability version for three equal volumes; then try to generalize to unequal volumes.}
%\snote{can i make the same statements for colding-minicozzi entropy?  might not be able to.}

Classical isoperimetric theory asks for the minimum total Euclidean surface area of $m$ disjoint volumes in $\R^{\adimn}$.  The case $m=1$ results in the Euclidean ball.  That is, a Euclidean ball has the smallest Euclidean surface area among all (measurable) sets of fixed Lebesgue measure.  The case $m=2$ is the Double Bubble Problem, solved in \cite{hutchings02,reichardt08}.  The case $m\geq3$ is still open, except for the special case $m=3,\adimn=2$ \cite{wichiramala04}.  As Hutchings writes on his website\footnote{\href{math.berkeley.edu/~hutching/pub/bubbles.html}{math.berkeley.edu/$\sim$hutching/pub/bubbles.html}} concerning the $m=3,\adimn=3$ case, ``The triple bubble problem in $\R^{3}$ currently seems hopeless without some brilliant new idea.''

Recent results in theoretical computer science, such as sharp hardness for the MAX-$m$-CUT problem \cite{khot07,isaksson11} motivate the above isoperimetric problem with Lebesgue measure replaced with the Gaussian measure \cite{isaksson11}.  Also, the ``plurality is stablest'' conjecture from social choice theory is closely related to such an isoperimetric problem.  This problem \cite{isaksson11} says that if votes are cast in an election between $m$ candidates, if every candidate has an equal chance of winning, and if no one person has a large influence on the outcome of the election, then taking the plurality is the most noise-stable way to determine the winner of the election.  That is, plurality is the voting method where the outcome is least likely to change due to independent, uniformly random changes to the votes.  The latter conjecture is a generalization of the ``majority is stablest conjecture'' \cite{mossel10}, which was proven using (a generalization of) the Gaussian $m=1$ case of the isoperimetric problem posed above.

In the Gaussian setting, for convenience, we include the complement of the $m$ volumes as a set itself.  That is, in the Gaussian setting, we ask for the minimum total Gaussian surface area of $m$ disjoint volumes in $\R^{\adimn}$ whose union is all of $\R^{\adimn}$.  The case $m=1$ is then vacuous.  The case $m=2$ results in two half spaces.  That is, a set $\Omega\subset\R^{\adimn}$ lying on one side of a hyperplane has the smallest Gaussian surface area $\int_{\partial\Omega}\gamma_{\sdimn}(x)dx$ among all (measurable) sets of fixed Gaussian measure $\int_{\Omega}\gamma_{\adimn}(x)dx$ \cite{sudakov74}.  Here, $\forall$ $k\geq1$, we define
\begin{flalign*}
\gamma_{k}(x)&\colonequals (2\pi)^{-k/2}e^{-\vnormt{x}^{2}/2},\qquad
\langle x,y\rangle\colonequals\sum_{i=1}^{\adimn}x_{i}y_{i},\qquad
\vnormt{x}^{2}\colonequals\langle x,x\rangle,\\
&\qquad\forall\,x=(x_{1},\ldots,x_{\adimn}),y=(y_{1},\ldots,y_{\adimn})\in\R^{\adimn}.
\end{flalign*}
This Gaussian isoperimetric result \cite{sudakov74} in the case $m=2$ has been elucidated and strengthened over the years \cite{borell85,ledoux94,ledoux96,bobkov97,burchard01,borell03,mossel15,mossel12,eldan13,mcgonagle15,barchiesi16}.
The case $m=3$ is the Gaussian Double Bubble Problem, solved when the sets $\Omega_{1},\Omega_{2},\Omega_{3}\subset\R^{\adimn}$ both have Gaussian measures close to $1/3$ \cite{corneli08}.  The result \cite{corneli08} requires the solution of the Double Bubble problem on a sphere of arbitrary dimension, so their methods seem difficult to apply to the case $m\geq3$.  Independent of the present work, the Gaussian case $m=3$ was recently resolved unconditionally in \cite{milman18a}.

Recent proofs of the Gaussian $m=2$ case \cite{mcgonagle15,barchiesi16} have used the calculus of variations techniques, originating with the related work of Colding and Minicozzi \cite{colding12a}.  We will focus on such calculus of variations techniques in this work, since they show the most promise for resolving the case $m\geq3$.  Such techniques have been applied to related problems in \cite{heilman15,heilman17}.  The work \cite{milman18a} also uses calculus of variations techniques, combined with a second-order matrix-valued differential inequality.  Instead of using volume-preserving variations of sets, the approach of \cite{milman18a} considers arbitrary perturbations of sets, allowing the evaluation of all minimizing sets for all measure restrictions, simultaneously.  The differential inequality infinitesimally ``pieces together'' minimizing sets of different measure restrictions, and ultimately solves the problem via a maximum principle.  In our approach, we instead mostly focus on volume-preserving variations of sets.

\begin{remark}
Unless otherwise stated, all Euclidean sets in this work are assumed to be Lebesgue measurable.  In Section \ref{intsec} below, we ignore technical issues such as non-compactness and regularity of sets, for didactic purposes.
\end{remark}

\subsection{Gaussian Isoperimetry for One Set}\label{intsec}

It follows from the work \cite{colding12a}, as recounted in \cite{mcgonagle15,barchiesi16}, that if $\Omega\subset\R^{\adimn}$ minimizes $\int_{\partial\Omega}\gamma_{\sdimn}(x)dx$ among all subsets of fixed Gaussian volume $\int_{\Omega}\gamma_{\adimn}(x)dx$, then $\exists$ $\lambda\in\R$ such that the boundary of $\Omega$ satisfies the following first variation condition:
\begin{equation}\label{zero0}
H(x)=\langle x,N(x)\rangle+\lambda,\qquad\forall\,x\in\partial\Omega.
\end{equation}
Here $N(x)$ is the exterior unit normal vector of $\Omega$ at $x\in\partial\Omega$ (satisfying $\vnormt{N(x)}=1$), and $H(x)$ is the mean curvature of $\partial\Omega$ at $x\in\partial\Omega$, i.e. the divergence of $N(x)$.  (See Section \ref{seccurvature} for more detailed definitions.)  Furthermore, the following second variation condition holds:
$$\forall\,f\colon\partial\Omega\to\R\,\,\mathrm{such}\,\,\mathrm{that}\,\,\int_{\partial\Omega}f(x)\gamma_{\sdimn}(x)dx=0,
\qquad\int_{\partial\Omega}f(x)Lf(x)\gamma_{\sdimn}(x)dx\leq0,$$
where $L$ is a second-order elliptic differential operator on $\partial\Omega$ (defined in \eqref{three4.5}), similar to the Ornstein-Uhlenbeck operator.

So, the problem of classifying the sets $\Omega$ of fixed Gaussian volume and minimal Gaussian surface area reduces to investigating the spectrum (and eigenfunctions) of a self-adjoint differential operator $L$ on $\partial\Omega$.  Fortunately, as shown in \cite{colding12a,mcgonagle15,barchiesi16}, there is a large class of eigenfunctions of $L$ with eigenvalue $1$.  That is,
\begin{equation}\label{zero1}
\forall\,v\in\R^{\adimn},\qquad L\langle v,N(x)\rangle=\langle v,N(x)\rangle,\qquad\forall\,x\in\partial\Omega.
\end{equation}
Combining this equality with the second variation condition, we see that
\begin{equation}\label{zero11}
\forall\,v\in\R^{\adimn}\,\,\mathrm{such}\,\,\mathrm{that}\,\,\int_{\partial\Omega}\langle v,N(x)\rangle\gamma_{\sdimn}(x)dx=0,
\qquad\int_{\partial\Omega}\langle v,N(x)\rangle^{2}\gamma_{\sdimn}(x)dx\leq0.
\end{equation}
If we define $V\colonequals\{v\in\R^{\adimn}\colon\int_{\partial\Omega}\langle v,N(x)\rangle\gamma_{\sdimn}(x)dx=0\}$, then \eqref{zero11} implies
$$\forall\,v\in V,\,\forall\,x\in\partial\Omega,\quad\langle v,N(x)\rangle=0.$$
Since $V\subset\R^{\adimn}$ is a linear subspace of dimension at least $\sdimn$, we conclude that $\partial\Omega$ consists of a set of parallel hyperplanes.  From \eqref{zero0}, $\partial\Omega$ must in fact consist of at most two hyperplanes, since $H(x)=0$ and $N(x)$ takes at most two values for all $x\in\partial\Omega$, so \eqref{zero0} gives the equation of at most two hyperplanes.  One can then deduce that in fact $\partial\Omega$ consists of a single hyperplane.

The key of the above argument was the implication \eqref{zero0} implies \eqref{zero1}.  That is, it is crucial to find a large class of eigenfunctions of $L$ with positive eigenvalues.  In the Gaussian setting, this argument seems to have first appeared in \cite{colding12a}, and it was used in several subsequent works, e.g. \cite{colding12,mcgonagle15,guang15,cheng14,cheng15,cheng16,barchiesi16,zhu16,heilman17}.  In the Euclidean setting, an investigation of the eigenfunctions of the second variation of minimal surfaces seems to have originated in the work of Simons \cite{simons68} on the stability of Euclidean minimal cones, reappearing elsewhere such as \cite{barbosa84} and \cite{hutchings02}.

One might ask why \eqref{zero1} would be expected to hold, and if such an identity is specific to the Gaussian setting.  Let $\Omega\subset\R^{\adimn}$, and let $v\in\R^{\adimn}$.  Suppose we translate $\Omega$ in the direction $tv$ for any $t>0$, forming the set $\Omega+tv$.  By formally taking a derivative at $t=0$, we see that $x\in\partial\Omega$ is translated to the point $x+t\langle v,N(x)\rangle N(x)+O_{x}(t^{2})$.  In this way, the function $x\mapsto\langle v,N(x)\rangle$ corresponds to an infinitesimal translation of the set $\Omega$ in Euclidean space.  The Gaussian measure is not translation invariant.  However, for any $v\in\R^{\adimn}$, the function $x\mapsto e^{-\vnorm{x}_{2}^{2}/2}$, $x\in\R^{\adimn}$ can be written as $x\mapsto e^{-\abs{\langle x,v\rangle}^{2}\vnorm{v}_{2}^{2}/2}e^{-\vnorm{x-\langle x,v\rangle v}_{2}^{2}/2}$.  And the second term in the product is invariant under translation by $v$.  Intuitively, this property of ``translation invariance up to one-dimensional distributions'' leads to the identity \eqref{zero1}, and also demonstrates that \eqref{zero1} seems unique to the Gaussian measure.  (Recall that (general) Gaussian measures are the unique rotation invariant product probability measures on Euclidean space.)
% x=<x,v>v+(x-<x,v>v)

In the Euclidean setting, infinitesimal translations and infinitesimal rotations also yield eigenfunctions of the corresponding operator $L$, albeit with eigenvalue $0$.  This observation was one key ingredient in the proof of the Euclidean Double Bubble Conjecture \cite{hutchings02}.

\subsection{Gaussian Isoperimetry of Multiple Sets}

It is an elementary but important observation that the above argument provides a nontrivial conclusion for multiple sets.  Suppose $\Omega_{1},\ldots\Omega_{m}\subset\R^{\adimn}$ are disjoint sets with fixed Gaussian volumes $\gamma_{\adimn}(\Omega_{1}),\ldots,\gamma_{\adimn}(\Omega_{m})$ with $\sum_{i=1}^{m}\gamma_{\adimn}(\Omega_{i})=1$ and of minimal Gaussian surface area
$$\sum_{1\leq i<j\leq m}\int_{(\partial\Omega_{i})\cap(\partial\Omega_{j})}\gamma_{\sdimn}(x)dx.$$
Then the above first and second variation arguments still hold, with suitable modifications (see Lemmas \ref{lemma24}, \ref{lemma21} and \ref{lemma28}).  For any $1\leq i<j\leq m$, there exists $\lambda_{ij}\in\R$ such that
\begin{equation}\label{zero0p}
H_{ij}(x)=\langle x,N_{ij}(x)\rangle+\lambda_{ij},\qquad\forall\,x\in(\partial\Omega_{i})\cap(\partial\Omega_{j}).
\end{equation}
As above, $N_{ij}(x)$ is the unit normal vector of $\Omega_{i}$ pointing into $\Omega_{j}$ at $x\in(\partial\Omega_{i})\cap(\partial\Omega_{j})$ and $H_{ij}(x)$ is the mean curvature of $\partial\Omega_{i}$ at $x\in(\partial\Omega_{i})\cap(\partial\Omega_{j})$, or the divergence of $N_{ij}(x)$.  (See Section \ref{seccurvature} and Definition \ref{defnote} for more detailed definitions.)  (See also Figure \ref{notationfig}.)  An appropriate generalization of the second variation formula now holds (see Lemma \ref{lemma28}).  For simplicity of exposition, we restrict our attention now to the second variation of linear functions of the normal vector:
\begin{flalign*}
&\forall\,v\in\R^{\adimn}\,\,\mathrm{such}\,\,\mathrm{that}\,\,\forall\,1\leq i\leq m,\,\int_{\partial\Omega_{i}}\sum_{j\in\{1,\ldots,m\}\colon j\neq i}\langle v,N_{ij}(x)\rangle\gamma_{\sdimn}(x)dx=0,\\
%&\mathrm{and}\,\,\mathrm{such}\,\,\mathrm{that}\,\,\forall\,1\leq i<j<k\leq m,\,\, f_{ij}(x)+f_{jk}(x)+f_{ki}(x)=0\,\,\forall\,x\in(\partial\Omega_{i})\cap(\partial\Omega_{j})\cap(\partial\Omega_{k}),\\
&\qquad\qquad\qquad\qquad\sum_{1\leq i<j\leq m}\int_{\partial\Omega_{i}}f(x)L_{ij}f(x)\gamma_{\sdimn}(x)dx\leq 0,
\end{flalign*}
where $L_{ij}$ is a second-order elliptic differential operator on $(\partial\Omega_{i})\cap(\partial\Omega_{j})$ defined in a similar way to $L$ (in \eqref{three4.5} or \eqref{one3.5n}).  Just as \eqref{zero0} implies \eqref{zero1}, \eqref{zero0p} implies
\begin{equation}\label{zero1p}
\forall\,v\in\R^{\adimn},\,\forall\,1\leq i<j\leq m,\qquad L_{ij}\langle v,N_{ij}(x)\rangle=\langle v,N_{ij}(x)\rangle,\qquad\forall\,x\in(\partial\Omega_{i})\cap(\partial\Omega_{j}).
\end{equation}
Combining this equality with the second variation condition, we see that
\begin{equation}\label{zero2p}
\begin{aligned}
&\forall\,v\in\R^{\adimn}\,\,\mathrm{such}\,\,\mathrm{that}\,\,\forall\,1\leq i\leq m,\,\int_{\partial\Omega_{i}}\sum_{j\in\{1,\ldots,m\}\colon j\neq i}\langle v,N_{ij}(x)\rangle\gamma_{\sdimn}(x)dx=0,\\
&\qquad\qquad\qquad\qquad\sum_{1\leq i<j\leq m}\int_{\partial\Omega_{i}}\langle v,N_{ij}(x)\rangle^{2}\gamma_{\sdimn}(x)dx\leq0.
\end{aligned}
\end{equation}

Let $V\colonequals\{v\in\R^{\adimn}\colon\int_{\partial\Omega_{i}}\sum_{j\in\{1,\ldots,m\}\colon j\neq i}\langle v,N_{ij}(x)\rangle\gamma_{\sdimn}(x)dx=0,\,\forall\,1\leq i\leq m\}$.  By \eqref{zero2p},
$$\forall\,v\in V,\,\forall\,1\leq i<j\leq m,\,\forall\,x\in(\partial\Omega_{i})\cap(\partial\Omega_{j}),\quad\langle v,N_{ij}(x)\rangle=0.$$
Since $V\subset\R^{\adimn}$ is a linear subspace of dimension at least $\sdimn-m+2$ (when $m\geq2$), after rotating $\Omega_{1},\ldots\Omega_{m}$, there exist $\Omega_{1}',\ldots,\Omega_{m}'\subset\R^{m-1}$ such that
$$\Omega_{i}=\Omega_{i}'\times\R^{\sdimn-m+2},\qquad\forall\,1\leq i\leq m.$$
That is, $\Omega_{1},\ldots\Omega_{m}$ are $(m-1)$-dimensional.

As in the case $m=2$, the implication \eqref{zero0p} implies \eqref{zero1p} is crucial for this argument.  In particular, this argument needs many eigenvectors of $(L_{ij})_{1\leq i<j\leq m}$.  This observation, that Colding-Minicozzi theory (i.e. the implication \eqref{zero0p} implies \eqref{zero1p}) applies to multiple sets, is the starting point of our investigation.

In certain cases, one can assert, as in \cite{colding12a}, that an eigenfunction of $(L_{ij})_{1\leq i<j\leq m}$ exists with eigenvalue greater than $1$.  This eigenfunction may have a nonzero integral on the boundary of some set $\Omega_{i}$ for some $1\leq i\leq m$, in which case the perturbation of sets induced by this eigenfunction does not preserve the Gaussian volumes of the sets $\Omega_{1},\ldots,\Omega_{m}$.  But this problem can be fixed by adding some other eigenfunctions of $(L_{ij})_{1\leq i<j\leq m}$, and using orthogonality of eigenfunctions with different eigenvalues in Lemma \ref{rkorth}.  This observation was also used in \cite{colding12a}.

We now state our main problem of interest more formally.

\begin{prob}[\embolden{Gaussian Multi-Bubble Problem}, {\cite{hutchings97,hutchings02,corneli08}}]\label{prob1}
Let $m\geq3$.  Fix $a_{1},\ldots,a_{m}>0$ such that $\sum_{i=1}^{m}a_{i}=1$.  Find measurable sets $\Omega_{1},\ldots\Omega_{m}\subset\R^{\adimn}$ with $\cup_{i=1}^{m}\Omega_{i}=\R^{\adimn}$ and $\gamma_{\adimn}(\Omega_{i})=a_{i}$ for all $1\leq i\leq m$ that minimize
$$\sum_{1\leq i<j\leq m}\int_{(\partial\Omega_{i})\cap(\partial\Omega_{j})}\gamma_{\sdimn}(x)dx,$$
subject to the above constraints.
\end{prob}

\begin{figure}[h]
\centering
\def\svgwidth{.4\textwidth}
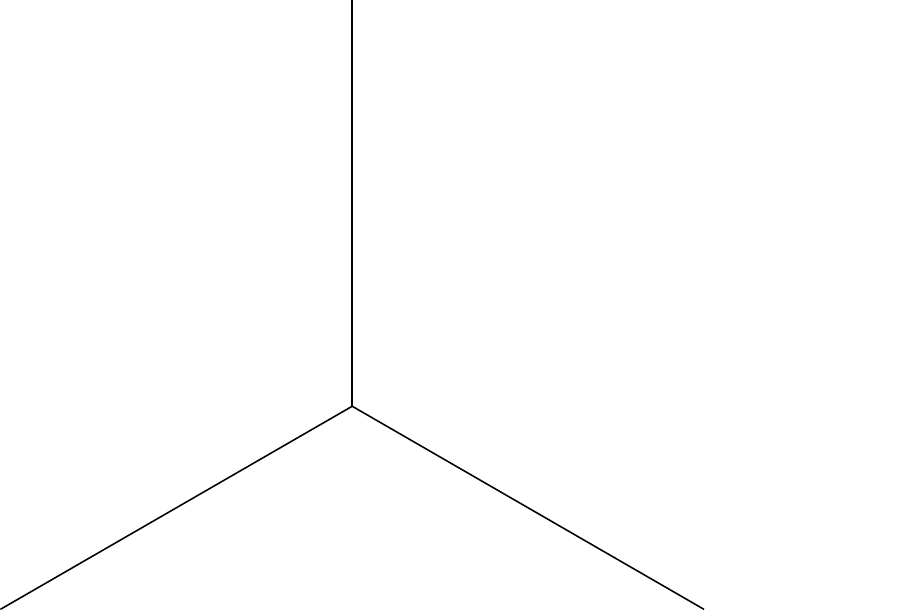
\caption{Optimal Sets for Conjecture \ref{conj0} in the case $m=3$, $\adimn=2$.}
\end{figure}

\begin{conj}[\embolden{Gaussian Multi-Bubble Conjecture} {\cite{hutchings02,corneli08,isaksson11}}]\label{conj0}
Let $\Omega_{1},\ldots\Omega_{m}\subset\R^{\adimn}$ minimize Problem \ref{prob1}.  Let $z_{1},\ldots,z_{m}\in\R^{\adimn}$ be the vertices of a regular simplex in $\R^{\adimn}$ centered at the origin.  Then $\exists$ $y\in\R^{\adimn}$ such that, for all $1\leq i\leq m$,
$$\Omega_{i}=y+\{x\in\R^{\adimn}\colon\langle x,z_{i}\rangle=\max_{1\leq j\leq m}\langle x,z_{j}\rangle\}.$$
\end{conj}

We sometimes refer to sets $\Omega_{1},\ldots,\Omega_{m}$ that minimize Problem \ref{prob1} as \textbf{Gaussian minimal bubbles}.  Our first result follows from the argument sketched above for \eqref{zero2p}.

\begin{theorem}[\embolden{Dimension Reduction for Gaussian Minimal Bubbles}]\label{thm0}
Suppose $\Omega_{1},\ldots\Omega_{m}\subset\R^{\adimn}$ minimize Problem \ref{prob1}.  Then there exists $0\leq \ell\leq m-1$ and there exist $\Omega_{1}',\ldots,\Omega_{m}'\subset\R^{\ell}$ such that, after rotating $\Omega_{1},\ldots,\Omega_{m}$, we have
$$\Omega_{i}=\Omega_{i}'\times\R^{\sdimn-\ell+1}.$$
Moreover $\ell$ can be chosen to be the dimension of the span of the following $m$ vectors in $\R^{\adimn}$
$$\int_{\Omega_{1}}x\gamma_{\adimn}(x)dx,\ldots,\int_{\Omega_{m}}x\gamma_{\adimn}(x)dx.$$
\end{theorem}

Theorem \ref{thm0} can be considered a Gaussian analogue of Hutchings' Symmetry Theorem for Euclidean minimizing bubbles \cite{hutchings97}.  In \cite[Theorem 2.6]{hutchings97}, it is shown that if $m\leq\sdimn$ sets in $\R^{\adimn}$ have fixed Euclidean volumes and minimum total Euclidean surface area, then there exists a linear subspace $P\subset\R^{\adimn}$ of dimension $(m-1)$ such that all of the sets are invariant under any isometry of $\R^{\adimn}$ that fixes $P$.  That is, all of the sets are symmetric across $P$.  In particular, when $m=2\leq\sdimn$, any minimal Euclidean double bubble is symmetric with respect to rotations along a fixed line.

The proof of Hutchings \cite[Theorem 2.6]{hutchings97} relies on an inductive version of Euclidean symmetrization together with the Ham Sandwich Theorem, for bisecting $m$ Euclidean sets with a hyperplane.  In particular, the proof uses translation invariance of Lebesgue measure.  So, an adaptation of the methods of \cite[Theorem 2.6]{hutchings97} to prove Theorem \ref{thm0} seems difficult if not impossible.  Moreover, Theorem \ref{thm0} deduces a translational symmetry for Gaussian bubble clusters, whereas \cite[Theorem 2.6]{hutchings97} deduces a rotational symmetry.  Indeed, the minimal Euclidean bubbles should not have any translational symmetry.

The following Theorem implies flatness of boundaries of minimal Gaussian partitions, or it improves on the dimension condition of Theorem \ref{thm0} by one dimension.

\begin{theorem}[\embolden{Structure Theorem Dichotomy}]\label{thm1}
Let $\Omega_{1},\ldots,\Omega_{m}$ minimize Problem \ref{prob1}.  Suppose Assumption \ref{as2} holds.  Then at least one of the following holds.
\begin{itemize}
\item $\forall$ $1\leq i<j\leq m$, there exist two hyperplanes $P_{ij},P_{ij}'\subset\R^{\adimn}$ such that
$$(\partial\Omega_{i})\cap(\partial\Omega_{j})\subset P_{ij}\cup P_{ij}'.$$
\item The sets $\Omega_{1},\ldots,\Omega_{m}$ are $(m-2)$-dimensional.  That is, there exists $0\leq \ell\leq m-2$ and there exist $\Omega_{1}',\ldots,\Omega_{m}'\subset\R^{\ell}$ such that, after rotating $\Omega_{1},\ldots,\Omega_{m}$, we have
$$\Omega_{i}=\Omega_{i}'\times\R^{\sdimn-\ell+1}.$$
\end{itemize}
\end{theorem}

Intuitively, Theorem \ref{thm1} should be sufficient to prove Conjecture \ref{conj0}, since $(m-2)$-dimensional sets should be ``unstable.''  For example, a small perturbation of the sets should always result in the span of $\int_{\Omega_{1}}x\gamma_{\adimn}(x)dx,\ldots,\int_{\Omega_{m}}x\gamma_{\adimn}(x)dx$ being $(m-1)$-dimensional, so Theorem \ref{thm0} should imply that the second case of Theorem \ref{thm1} does not occur.  We are unable to turn this intuition into a proof.

The following technical assumption is needed to prove the almost orthogonality of eigenfunctions of $(L_{ij})_{1\leq i<j\leq m}$ in the proof of Theorem \ref{thm1}.  Assumption \ref{as2} says that a neighborhood of the singular set has small Gaussian area on a sequence of annuli going to infinity.
\begin{assumption}[\embolden{Polynomial Volume Growth of Singular Set}]\label{as2}
For any $\epsilon,r>0$, define
$$C\colonequals\bigcup_{1\leq i<j<k\leq m}(\partial\Omega_{i})\cap(\partial\Omega_{j})\cap(\partial\Omega_{k}).$$
$$C_{\epsilon,r}\colonequals\{x\in\R^{\adimn}\colon\exists\,c\in C\,\,\mathrm{such}\,\,\mathrm{that}\,\, \vnorm{x-c}<\epsilon\,\,\mathrm{and}\,\,r<\vnorm{c}<r+1\}.$$
So, $C_{\epsilon,r}$ is the $\epsilon$-neighborhood of the part of $C$ lying in the annulus with radii $r$ and $r+1$.

We assume that: for any $\epsilon>0$, there exist a $0<r_{1}<r_{2}<\cdots$ such that $\lim_{\ell\to\infty}r_{\ell}=\infty$ and for any $\ell\geq1$, there exists $\eta_{\ell,\epsilon}\in C_{0}^{\infty}(\R^{\adimn})$ such that $\eta_{\ell,\epsilon}=0$ on $C_{\epsilon,r_{j}}$, $0\leq\eta_{\ell,\epsilon}\leq1$ everywhere, $\eta_{\ell,\epsilon}=1$ on $C_{2\epsilon,r_{j}}^{c}$, and
$$\lim_{\ell\to\infty}\limsup_{\epsilon\to0^{+}}\sum_{1\leq i<j\leq m}\int_{(\partial\Omega_{i})\cap(\partial\Omega_{j})}\vnorm{\nabla \eta_{\ell,\epsilon}}^{2}\gamma_{\sdimn}(x)dx=0.$$
\end{assumption}
For example, if $\adimn=2$ and if $C$ is a finite set of points, then Assumption \ref{as2} automatically holds, since the set $\{x\in C\colon r\leq \vnorm{x}\leq r+1\}$ is empty for all large $r$.  It is unclear whether or not Assumption \ref{as2} can be proven to hold a priori for any minimizers of Problem \ref{prob1}.

Assumption \ref{as2} is only used in the proof of Lemma \ref{rkorth} to prove the orthogonality of two eigenfunctions with different eigenvalues. % In fact, we only that the quantity $\gamma_{\sdimn-1}(\{x\in C\colon r\leq \vnorm{x}\leq r+1\})$ becomes sufficiently small when $r$ is large, since we only use Lemma \ref{rkorth} when $\epsilon=1/2$.
Orthogonality of eigenfunctions with different eigenvalues is nontrivial in this setting since the surface $\Sigma\colonequals\cup_{1\leq i<j\leq m}(\partial\Omega_{i})\cap(\partial\Omega_{j})$ is non-compact with nonempty boundary.

\begin{figure}[h]
\centering
\def\svgwidth{.4\textwidth}
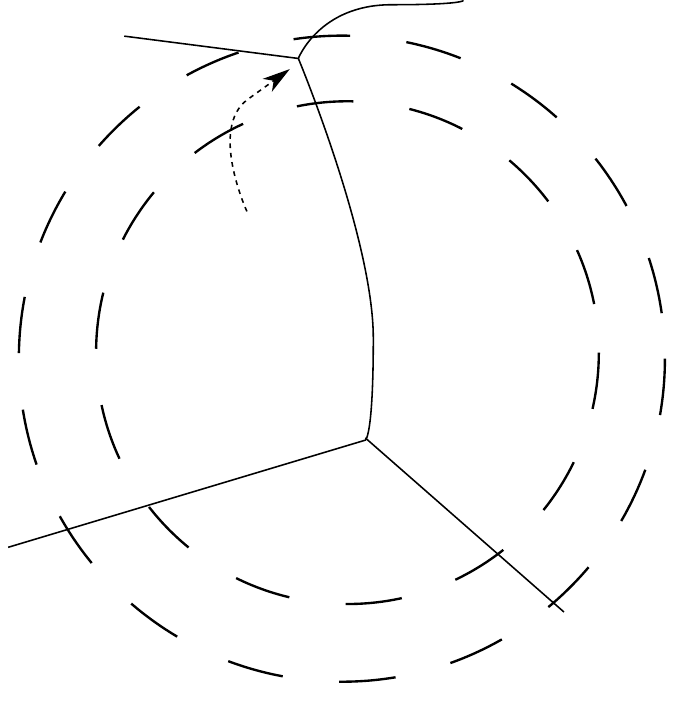
\caption{Depiction of the set $C_{\epsilon,r}$, an annular neighborhood of the singular set $C$.}
\end{figure}

In light of Theorem \ref{thm0}, we say that sets $\Omega_{1},\ldots,\Omega_{m}$ are \textbf{$\ell$-dimensional} if the dimension of the span of $\int_{\Omega_{1}}x\gamma_{\adimn}(x)dx,\ldots,\int_{\Omega_{m}}x\gamma_{\adimn}(x)dx$ is $\ell$.

In the case $m=3$, Theorem \ref{thm0} says the sets $\Omega_{1},\Omega_{2},\Omega_{3}$ minimizing Problem \ref{prob1} are at most two-dimensional.  So, without loss of generality, if $m=3$ in Conjecture \ref{conj0}, we may assume that $\adimn=2$.  Theorem \ref{thm1} then implies that $\Omega_{1},\Omega_{2},\Omega_{3}\subset\R^{2}$ have flat (one-dimensional) boundaries (contained in at most $6$ lines), or the sets $\Omega_{1},\Omega_{2},\Omega_{3}$ are one-dimensional (with boundaries contained in at most $6$ lines, by \eqref{zero0p}).  Given these reductions, one can e.g. check a finite number of cases to conclude that Conjecture \ref{conj0} holds when $m=3$ (if Assumption \ref{as2} holds.)

\begin{cor}[\embolden{Gaussian Double Bubble Problem}]\label{cor1}
If any $\Omega_{1},\ldots,\Omega_{m}$ minimizing Problem \ref{prob1} satisfy Assumption \ref{as2}, then Conjecture \ref{conj0} holds when $m=3$.
\end{cor}
Corollary \ref{cor1} was recently proven independently by \cite{milman18a}, without the need for Assumption \ref{as2}.  We emphasize that the proof of Corollary \ref{cor1} has a different strategy than the proof of the main result of \cite{milman18a}.  In particular, the proof of Corollary \ref{cor1} proceeds along the lines of \cite{colding12a,zhu16,heilman17}.

In the case $m=4$, Theorem \ref{thm0} says the sets $\Omega_{1},\ldots,\Omega_{4}$ minimizing Problem \ref{prob1} are at most three-dimensional.  So, without loss of generality, if $m=4$ in Conjecture \ref{conj0}, we may assume that $\adimn=3$.  If these sets are three-dimensional, then the first case of Theorem \ref{thm1} implies that the sets have flat boundaries.  If the sets are at most two-dimensional, then we adapt the argument of \cite{milman18a} to show that the sets must have flat boundaries.   So, in any case, the sets $\Omega_{1},\ldots,\Omega_{4}$ have flat boundaries, and the boundaries are contained in at most $8$ planes in $\R^{3}$, by \eqref{zero0p}.  Given these reductions, one can check various cases to conclude that Conjecture \ref{conj0} holds when $m=4$ (if Assumption \ref{as2} holds.)

\begin{cor}[\embolden{Gaussian Triple Bubble Problem}]\label{cor2}
If any $\Omega_{1},\ldots,\Omega_{m}$ minimizing Problem \ref{prob1} satisfy Assumption \ref{as2}, then Conjecture \ref{conj0} holds when $m=4$.
\end{cor}
To our knowledge, this is the first known proof of any triple bubble problem in arbitrary dimension, albeit conditional on Assumption \ref{as2}.  (For an update, see Section \ref{secsub}.)

The second case of Theorem \ref{thm1} corresponds to the second variation operator having an eigenvalue larger than $1$, which really should not happen (see \eqref{one3.5n}, \eqref{seven00}).  An eigenvalue larger than one indicates that something other than an infinitesimal translation can decrease the Gaussian surface area of sets minimizing Problem \ref{prob1}.  And this should not happen since Conjecture \ref{conj0} predicts that the sets of minimal Gaussian surface area (for different volume restrictions) should be translations of each other.

As noted already in \eqref{zero2p}, the more eigenfunctions with positive eigenvalues there are for the second variation operator \eqref{seven00}, the more control can be placed on the structure of the sets minimizing Problem \ref{prob1}.  That is, more eigenfunctions translates to flatness or lower-dimensionality of the minimal sets.  The general scheme of our arguments is then to look for as many eigenfunctions as possible of the second variation operator, as in \eqref{zero1p}.  In general it seems impossible to find explicit eigenfunctions beyond those described in \eqref{zero1p}, so we often simply prove the existence of such eigenfunctions.  This was a key step in the arguments of \cite{colding12a}.

\begin{remark}
In order to prove Conjecture \ref{conj0} for $m>4$, one might try to extend Theorem \ref{thm1} to higher dimensions.
\end{remark}

%\subsection{Noise Stability}
%
%Unfortunately, if we repeat the above first and second variation arguments with noise stability in place of Gaussian surface area, it seems difficult to explicitly identify eigenfunctions of the second variation operator.  That is, the analogue of \eqref{zero1p} no longer holds.  In fact, the second variation operator for noise stability is no longer a differential operator, but an integral operator (see Lemma \ref{lemma6}).  Despite these difficulties, an approximate version of \eqref{zero1p} still holds.  Since the analogue of \eqref{zero1p} does not have an exact equality, the conclusion one can make for sets optimizing noise stability is weaker than the case of Gaussian surface area.  Nevertheless, the analogue of \eqref{zero1p} can still place nontrivial geometric constraints on sets optimizing noise stability.

\subsection{Organization}

\begin{itemize}
\item Theorem \ref{thm0} is proven in Section \ref{secred}.
\item Theorem \ref{thm1} is proven in Section \ref{secdi}.
\item Corollary \ref{cor1} is proven in Section \ref{secdoub}.
\item Corollary \ref{cor2} is proven in Section \ref{sectrip}.
\end{itemize}
Previous sections cover preliminary material.  Section \ref{secmilman} covers results from \cite{milman18a}.  Section \ref{secconc} provides some concluding remarks.

\subsection{Subsequent Work}\label{secsub}

After uploading the first version of this work to the arXiv on Friday May 25, 2018, Milman and Neeman \cite{milman18b} uploaded their paper to the arXiv on Monday May 28, 2018, solving Conjecture \ref{conj0} unconditionally.  Nevertheless, our work and theirs were completed at essentially the same time.  One observation of theirs that did not appear in our work is that the second case of Theorem \ref{thm1} does not actually occur (see \cite[Proposition 7.6]{milman18b}).  Using our terminology and notation, in the second case of Theorem \ref{thm1}, there exists a piecewise constant function $F$ on $\cup_{i=1}^{m}\partial\Omega_{i}'$ whose second variation is negative (see Lemma \ref{lemma23}).  The observation of \cite[Proposition 7.6]{milman18b} is that the function $x_{\adimn}F$ on $\cup_{i=1}^{m}\partial\Omega_{i}$ also has negative second variation, and it is automatically (Gaussian) volume-preserving, contradicting the minimality of $\Omega_{1},\ldots,\Omega_{m}$.  Therefore, the second case of Theorem \ref{thm1} cannot occur, i.e. only the first case occurs.

\section{Preliminaries and Notation}\label{secpre}

We say that $\Sigma\subset\R^{\adimn}$ is an $\sdimn$-dimensional $C^{\infty}$ manifold with boundary if $\Sigma$ can be locally written as the graph of a $C^{\infty}$ function on a relatively open subset of $\{(x_{1},\ldots,x_{\sdimn})\in\R^{\sdimn}\colon x_{\sdimn}\geq0\}$.  For any $(\adimn)$-dimensional $C^{\infty}$ manifold $\Omega\subset\R^{\adimn}$ such that $\partial\Omega$ itself has a boundary, we denote
\begin{equation}\label{c0def}
\begin{aligned}
C_{0}^{\infty}(\Omega;\R^{\adimn})
&\colonequals\{f\colon \Omega\to\R^{\adimn}\colon f\in C^{\infty}(\Omega;\R^{\adimn}),\, f(\partial\partial \Omega)=0,\\
&\qquad\qquad\qquad\exists\,r>0,\,f(\Omega\cap(B(0,r))^{c})=0\}.
\end{aligned}
\end{equation}
We also denote $C_{0}^{\infty}(\Omega)\colonequals C_{0}^{\infty}(\Omega;\R)$.  We let $\mathrm{div}$ denote the divergence of a vector field in $\R^{\adimn}$.  For any $r>0$ and for any $x\in\R^{\adimn}$, we let $B(x,r)\colonequals\{y\in\R^{\adimn}\colon\vnormt{x-y}\leq r\}$ be the closed Euclidean ball of radius $r$ centered at $x\in\R^{\adimn}$.  Here $\partial\partial\Omega$ refers to the $(\sdimn-1)$-dimensional boundary of $\Omega$.

\begin{definition}[\embolden{Reduced Boundary}]
A measurable set $\Omega\subset\R^{\adimn}$ has \embolden{locally finite surface area} if, for any $r>0$,
$$\sup\left\{\int_{\Omega}\mathrm{div}(X(x))dx\colon X\in C_{0}^{\infty}(B(0,r),\R^{\adimn}),\, \sup_{x\in\R^{\adimn}}\vnormt{X(x)}\leq1\right\}<\infty.$$
Equivalently, $\Omega$ has locally finite surface area if $\nabla 1_{\Omega}$ is a vector-valued Radon measure such that, for any $x\in\R^{\adimn}$, the total variation
$$
\vnormt{\nabla 1_{\Omega}}(B(x,1))
\colonequals\sup_{\substack{\mathrm{partitions}\\ C_{1},\ldots,C_{m}\,\mathrm{of}\,B(x,1) \\ m\geq1}}\sum_{i=1}^{m}\vnormt{\nabla 1_{\Omega}(C_{i})}
$$
is finite \cite{cicalese12}.  If $\Omega\subset\R^{\adimn}$ has locally finite surface area, we define the \embolden{reduced boundary} $\redb \Omega$ of $\Omega$ to be the set of points $x\in\R^{\adimn}$ such that
$$N(x)\colonequals-\lim_{r\to0^{+}}\frac{\nabla 1_{\Omega}(B(x,r))}{\vnormt{\nabla 1_{\Omega}}(B(x,r))}$$
exists, and it is exactly one element of $S^{\sdimn}\colonequals\{x\in\R^{\adimn}\colon\vnorm{x}=1\}$.
\end{definition}

\begin{lemma}[\embolden{Existence}, { \cite[Theorem VI.2]{almgren76}}]\label{lemma51p}
There exist measurable $\Omega_{1},\ldots,\Omega_{m}\subset\R^{\adimn}$ minimizing Problem \ref{prob1}.
\end{lemma}
For more detail on this proof, see e.g. \cite[Theorem 4.1]{milman18a}.

\begin{lemma}[\embolden{Regularity}, {\cite[Theorem 1.3]{colombo17}, \cite[Theorem 4.1, Corollary 4.4]{milman18a}}]\label{lemma52.6}
Let $\Omega_{1},\ldots\Omega_{m}$ minimize problem \ref{prob1}.  Then Assumption \ref{as1} holds.
\end{lemma}

Let $Y_{1}\subset\R^{2}$ denote three half-lines meeting at a single point with $120$-degree angles between them.  Let $T'\subset\R^{3}$ be the one-dimensional boundary of a regular tetrahedron centered at the origin, and let $T_{2}\subset\R^{3}$ be the cone generated by $T'$, so that $T_{2}=\{rx\in\R^{3}\colon r\geq0,\, x\in T'\}$.

\begin{assumption}\label{as1}
The sets $\Omega_{1},\ldots\Omega_{m}\subset\R^{\adimn}$ satisfy the following conditions.  First, $\cup_{i=1}^{m}\Omega_{i}=\R^{\adimn}$, $\sum_{i=1}^{m}\gamma_{\adimn}(\Omega_{i})=1$.  Also, $\Sigma\colonequals\cup_{i=1}^{m}\partial\Omega_{i}$ can be written as the disjoint union $M_{\sdimn}\cup M_{\sdimn-1}\cup M_{\sdimn-2}\cup M_{\sdimn-3}$ where $0<\alpha<1$ and
\begin{itemize}
\item[(i)] $M_{\sdimn}$ is a locally finite union of embedded $C^{\infty}$ $\sdimn$-dimensional manifolds.
\item[(ii)] $M_{\sdimn-1}$ is a locally finite union of embedded $C^{\infty}$ $(\sdimn-1)$-dimensional manifolds, near which $\Sigma$ is locally diffeomorphic to $Y_{1}\times\R^{\sdimn-1}$.
\item[(iii)] $M_{\sdimn-2}$ is a locally finite union of embedded $C^{1,\alpha}$ $(\sdimn-2)$-dimensional manifolds, near which $\Sigma$ is locally diffeomorphic to $T_{2}\times\R^{\sdimn-2}$.
\item[(iv)] $M_{\sdimn-3}$ is relatively closed, $(\sdimn-3)$-rectifiable, with locally finite $(n-3)$-dimensional Hausdorff measure.
\end{itemize}
%And $\exists$ $\epsilon>0$ such that, for all $1\leq i\leq m$,
%$$\int_{\redb\Omega_{i}}\sup_{y\in B(x,\epsilon)}(1+\vnormt{y})\gamma_{\sdimn}(x)dx<\infty.$$
\end{assumption}

Below, when $\Sigma_{ij}\colonequals(\redb\Omega_{i})\cap(\redb\Omega_{j})$ for some $1\leq i<j\leq m$, we denote $\redS_{ij}\colonequals M_{\sdimn-2}\cap(\partial\Omega_{i})\cap(\partial\Omega_{j})$, where $M_{\sdimn-2}$ is defined in Assumption \ref{as1}.

%\begin{proof}
%Consider a sequence $\Omega_{1},\Omega_{2},\ldots\subset\R^{\adimn}$ of sets of locally finite surface area such that $\lim_{m\to\infty} \int_{\partial^{*}\Omega_{m}}$ is equal to the infimum of the Gaussian surface area over all sets of locally finite surface area.  Then, for any bounded open set $C$, $\sup_{m\geq1}\int_{(\partial^{*} \Omega_{m})\cap C}dx<\infty$, the compactness theorem for BV functions \cite[Theorem 3.23]{ambrosio00} ensures the existence of a Borel set $\Omega\subset\R^{\adimn}$ such that, there exists a subsequence of positive integers $m_{1}<m_{2}<\cdots$ such that $1_{\Omega_{m_{1}}},1_{\Omega_{m_{2}}},\ldots$ converges to $1_{\Omega}$ locally in $L_{1}(\R^{\adimn})$.
%From the Divergence Theorem,
%\begin{flalign*}
%\int_{\redA}\gamma_{\sdimn}(x)dx
%&=\sqrt{2\pi}\sup\Big\{\int_{\Omega}[\mathrm{div}(X(x))-\langle X(x),x\rangle]\gamma_{\sdimn}(x)dx\\
%&\quad\qquad\qquad\qquad\qquad\colon X\in C_{0}^{\infty}(B(0,r),\R^{\adimn}),\, r>0,\,\sup_{x\in\R^{\adimn}}\vnormt{X(x)}\leq1\Big\}.
%\end{flalign*}
%Therefore, $\Omega\mapsto\int_{\redA}\gamma_{\adimn}(x)dx$ is lower
%semicontinuous with respect to the local $L_{1}$ convergence of sets.  That is, $\int_{\redA}\gamma_{\adimn}(x)dx \leq\liminf_{p\to\infty}\int_{\redb \Omega_{m_{p}}}\gamma_{\sdimn}(x)dx$.  Therefore, $\int_{\redA}\gamma_{\adimn}(x)dx\leq \int_{\partial^{*}D}\gamma_{\adimn}(x)dx$ for any set $D\subset\R^{\adimn}$ of locally finite surface area, as desired.
%\end{proof}

\subsection{Submanifold Curvature}\label{seccurvature}

Here we cover some basic definitions from differential geometry of submanifolds of Euclidean space.

Let $\nabla$ denote the standard Euclidean connection, so that if $X,Y\in C_{0}^{\infty}(\R^{\adimn},\R^{\adimn})$, if $Y=(Y_{1},\ldots,Y_{\adimn})$, and if $u_{1},\ldots,u_{\adimn}$ is the standard basis of $\R^{\adimn}$, then $\nabla_{X}Y\colonequals\sum_{i=1}^{\adimn}(X (Y_{i}))u_{i}$.  Let $N$ be the outward pointing unit normal vector of an $\sdimn$-dimensional orientable hypersurface $\Sigma\subset\R^{\adimn}$.  For any vector $x\in\Sigma$, we write $x=x^{T}+x^{N}$, so that $x^{N}\colonequals\langle x,N\rangle N$ is the normal component of $x$, and $x^{T}$ is the tangential component of $x\in\Sigma$.

Let $e_{1},\ldots,e_{\sdimn}$ be a (local) orthonormal frame of $\Sigma\subset\R^{\adimn}$.  That is, for a fixed $x\in\Sigma$, there exists a neighborhood $U$ of $x$ such that $e_{1},\ldots,e_{\sdimn}$ is an orthonormal basis for the tangent space of $\Sigma$, for every point in $U$ \cite[Proposition 11.17]{lee03}.

Define the \embolden{mean curvature} of $\Sigma$ by
\begin{equation}\label{three0.5}
H\colonequals\mathrm{div}(N)=\sum_{i=1}^{\sdimn}\langle\nabla_{e_{i}}N,e_{i}\rangle.
\end{equation}

Define the \embolden{second fundamental form} $A=(a_{ij})_{1\leq i,j\leq\sdimn}$ of $\Sigma$ so that
\begin{equation}\label{three1}
a_{ij}=\langle\nabla_{e_{i}}e_{j},N\rangle,\qquad\forall\,1\leq i,j\leq \sdimn.
\end{equation}
Compatibility of the Riemannian metric says $a_{ij}=\langle\nabla_{e_{i}}e_{j},N\rangle=-\langle e_{j},\nabla_{e_{i}}N\rangle+ e_{i}\langle N,e_{j}\rangle=-\langle e_{j},\nabla_{e_{i}}N\rangle$, $\forall$ $1\leq i,j\leq \sdimn$.  So, multiplying by $e_{j}$ and summing this equality over $j$ gives
\begin{equation}\label{three2}
\nabla_{e_{i}}N=-\sum_{j=1}^{\sdimn}a_{ij}e_{j},\qquad\forall\,1\leq i\leq \sdimn.
\end{equation}

%Given a vector $v\in\R^{\adimn}$, define $v^{N}\colonequals \langle v,N\rangle N$, and define $v^{T}\colonequals v-v^{N}$, so that $v=v^{N}+v^{T}$.

Using $\langle\nabla_{N}N,N\rangle=0$,
\begin{equation}\label{three4}
H\stackrel{\eqref{three0.5}}{=}\sum_{i=1}^{\sdimn}\langle \nabla_{e_{i}} N,e_{i}\rangle
\stackrel{\eqref{three2}}{=}-\sum_{i=1}^{\sdimn}a_{ii}.
\end{equation}

For any vector field $X\in C_{0}^{\infty}(\R^{\adimn},\R^{\adimn})$, define the tangential divergence of $X$ on $\Sigma$ by
$$\mathrm{div}_{\tau}X\colonequals\sum_{i=1}^{\sdimn}\langle \nabla_{e_{i}}X,e_{i}\rangle.$$
When $\Sigma\colonequals\partial\Omega$ itself has a boundary that is a $C^{\infty}$ $(\sdimn-1)$-dimensional manifold, we let $\nu$ denote the unit normal of $\partial\Sigma$ pointing exterior to $\Sigma$.

\begin{remark}[{\cite[Page 6]{barchiesi16}}]\label{rk10}
Let $\Sigma$ be a $C^{\infty}$ $\sdimn$-dimensional manifold with boundary, and let $\redS$ denote the $(\sdimn-1)$-dimensional reduced boundary of $\Sigma$.  Assume that $\redS$ is also a $C^{\infty}$ manifold.  For any $x\in\redS$, let $\nu(x)\in\R^{\adimn}$ denote the exterior pointing normal vector to $\redS$ at $x$. The divergence theorem for hypersurfaces says, for any vector field $X\in C_{0}^{\infty}(\R^{\adimn},\R^{\adimn})$,
$$\int_{\Sigma}\mathrm{div}_{\tau}X(x)dx
=\int_{\Sigma}H(x)\langle X(x),N(x)\rangle dx
+\int_{\redS}\langle X,\nu\rangle dx.$$
\end{remark}
\begin{proof}
Write $X=\langle X,N\rangle N+(X-\langle X,N\rangle N)$.  Then by the usual divergence theorem and
\begin{flalign*}
\int_{\Sigma}\mathrm{div}_{\tau}Xdx
&=\int_{\Sigma}\mathrm{div}_{\tau}(\langle X,N\rangle N)+\mathrm{div}_{\tau}(X-\langle X,N\rangle N)dx\\
&=\int_{\Sigma}\langle X,N\rangle \mathrm{div}_{\tau}(N)+\langle N,\nabla\langle X,N\rangle\rangle dx+\int_{\redS}\langle (X-\langle X,N\rangle N),\nu\rangle dx\\
&\stackrel{\eqref{three4}}{=}\int_{\Sigma}H\langle X,N\rangle dx+\int_{\redS}\langle X,\nu\rangle dx.
\end{flalign*}

\end{proof}

\subsection{Colding-Minicozzi Theory for Mean Curvature Flow}

The Colding-Minicozzi theory \cite{colding12a,colding12} focuses on orientable $\sdimn$-dimensional $C^{\infty}$ hypersurfaces $\Sigma$ with $\partial\Sigma=\emptyset$ satisfying
\begin{equation}\label{one2}
H(x)=\langle x, N(x)\rangle,\qquad\forall\,x\in\Sigma.
\end{equation}
Below, we will often omit the $x$ arguments of $H$ and $N$ for brevity.  Here $H$ is chosen so that, if $r>0$, then the surface $r S^{\sdimn}$ satisfies $H(x)=n/r$ for all $x\in r S^{\sdimn}$.  A hypersurface $\Sigma$ satisfying \eqref{one2} is called a \embolden{self-shrinker}, since it is self-similar under the mean curvature flow.  Examples of self-shrinkers include a hyperplane through the origin, the sphere $\sqrt{\sdimn}S^{\sdimn}$, or more generally, round cylinders $\sqrt{k} S^{k}\times S^{\sdimn -k}$, where $0\leq k\leq n$, and also cones with zero mean curvature.

A key aspect of the Colding-Minicozzi theory is the study of eigenfunctions of\ the differential operator $L$, defined for any $C^{\infty}$ function $f\colon\Sigma\to\R$ by
\begin{equation}\label{three4.5}
L f\colonequals \Delta f-\langle x,\nabla f\rangle+f+\vnormt{A}^{2}f.
\end{equation}
\begin{equation}\label{three4.3}
\mathcal{L}f\colonequals \Delta f-\langle x,\nabla f\rangle.
\end{equation}
Note that there is a factor of $2$ difference between our definition of $L$ and the definition of $L$ in \cite{colding12a}.
Here $e_{1},\ldots,e_{\sdimn}$ is a (local) orthonormal frame for an orientable $C^{\infty}$ $\sdimn$-dimensional hypersurface $\Sigma\subset\R^{\adimn}$ with $\redS=\emptyset$, $\Delta\colonequals\sum_{i=1}^{\sdimn}\nabla_{e_{i}}\nabla_{e_{i}}$ be the Laplacian associated to $\Sigma$, $\nabla\colonequals\sum_{i=1}^{\sdimn}e_{i}\nabla_{e_{i}}$ is the gradient associated to $\Sigma$, $A=A_{x}$ is the second fundamental form of $\Sigma$ at $x$, and $\vnormt{A_{x}}^{2}$ is the sum of the squares of the entries of the matrix $A_{x}$.  Let $\mathrm{div}_{\tau}\colonequals\sum_{i=1}^{\sdimn}\nabla_{e_{i}}\langle\cdot,e_{i}\rangle$ be the (tangential) divergence of a vector field on $\Sigma$.  Note that $L$ is an Ornstein-Uhlenbeck-type operator.  In particular, if $\Sigma$ is a hyperplane, then $A_{x}=0$ for all $x\in\Omega$, so $L$ is exactly the usual Ornstein-Uhlenbeck operator, plus the identity map.  (More detailed definitions will be given in Section \ref{seccurvature} below.)

\begin{lemma}[\embolden{Linear Eigenfunction of $L$}, {\cite{mcgonagle15,barchiesi16}} {\cite[Lemma 4.2]{heilman17}}]\label{lemma45}
Let $\Sigma\subset\R^{\adimn}$ be an orientable $C^{\infty}$ $\sdimn$-dimensional hypersurface.  Let $\scon\in\R$.  Suppose
\begin{equation}\label{three0n}
H(x)=\langle x,N\rangle+\scon,\qquad\forall\,x\in\Sigma.
\end{equation}
Let $v\in\R^{\adimn}$.  Then
$$L\langle v,N\rangle=\langle v,N\rangle.$$
\end{lemma}

\section{First and Second Variation}

We will apply the calculus of variations to solve Problem \ref{prob1}. Here we present the rudiments of the calculus of variations.

Some of the results in this section are well known to experts in the calculus of variations, and many of these results were re-proven in \cite{barchiesi16}, or adapted from \cite{hutchings02}.

Let $\Omega\subset\R^{\adimn}$ be an $(\adimn)$-dimensional $C^{2}$ submanifold with reduced boundary $\Sigma\colonequals\redb \Omega$.  Let $N\colon\redA\to S^{\sdimn}$ be the unit exterior normal to $\redA$.  Let $X\colon\R^{\adimn}\to\R^{\adimn}$ be a vector field.

Let $\mathrm{div}$ denote the divergence of a vector field.  We write $X$ in its components as $X=(X_{1},\ldots,X_{\adimn})$, so that $\mathrm{div}X=\sum_{i=1}^{\adimn}\frac{\partial}{\partial x_{i}}X_{i}$.  Let $\Psi\colon\R^{\adimn}\times(-1,1)\to\R^{\adimn}$ such that
\begin{equation}\label{nine2.3}
\Psi(x,0)=x,\qquad\qquad\frac{d}{ds}\Psi(x,s)=X(\Psi(x,s)),\quad\forall\,x\in\R^{\adimn},\,s\in(-1,1).
\end{equation}
For any $s\in(-1,1)$, let $\Omega^{(s)}\colonequals\Psi(\Omega,s)$.  Note that $\Omega^{(0)}=\Omega$.  Let $\Sigma^{(s)}\colonequals\redb\Omega^{(s)}$, $\forall$ $s\in(-1,1)$.
\begin{definition}
We call $\{\Omega^{(s)}\}_{s\in(-1,1)}$ as defined above a \embolden{variation} of $\Omega\subset\R^{\adimn}$.  We also call $\{\Sigma^{(s)}\}_{s\in(-1,1)}$ a \embolden{variation} of $\Sigma=\redb\Omega$.
\end{definition}

\begin{lemma}[\embolden{First Variation}]\label{lemma10}  Let $X\in C_{0}^{\infty}(\R^{\adimn},\R^{\adimn})$.  Let $f(x)=\langle X(x),N(x)\rangle$ for any $x\in\redA$.  Then
\begin{equation}\label{nine1}
\frac{d}{ds}|_{s=0}\gamma_{\adimn}(\Omega^{(s)})=\int_{\redA}f(x) \gamma_{\adimn}(x)dx.
\end{equation}
\begin{equation}\label{nine2}
\frac{d}{ds}|_{s=0}\int_{\Sigma^{(s)}}\gamma_{\sdimn}(x)dx
=\int_{\Sigma}(H(x)-\langle N(x),x\rangle)f(x)\gamma_{\sdimn}(x)dx
+\int_{\redS}\langle X,\nu\rangle\gamma_{\sdimn}(x)dx.
\end{equation}
\end{lemma}
\begin{proof}
We first prove \eqref{nine1}.  Let $J\Psi(x,s)\colonequals\abs{\mathrm{det}(D\Psi(x,s))}$ be the Jacobian determinant of $\Psi$ $\forall$ $x\in\R^{\adimn}$, $\forall$ $s\in(-1,1)$.  Then \cite[Equation (2.28)]{chokski07} says
\begin{equation}\label{nine1.0}
\frac{d}{ds}|_{s=0}J\Psi(x,s)=\mathrm{div}X(x),\qquad\forall\,x\in\redA.
\end{equation}
So, the Chain Rule, $J\Psi(x,0)=1$ (which follows by \eqref{nine2.3}), and \eqref{nine2.3} imply
\begin{equation}\label{nine1.01}
\begin{aligned}
\frac{d}{ds}|_{s=0}\gamma_{\adimn}(\Omega^{(s)})
&=\frac{d}{ds}|_{s=0}\int_{\Omega}J\Psi(x,s)\gamma_{\adimn}(\Psi(x,s))dx
=\int_{\Omega}(\mathrm{div}(X(x))-\langle X(x),x\rangle)d\gamma_{\adimn}(x)\\
&=\int_{\Omega}\mathrm{div}(X(x)\gamma_{\adimn}(x))dx
=\int_{\redA}\langle X(x),N(x)\rangle\gamma_{\adimn}(x)dx.
\end{aligned}
\end{equation}
In the last line, we used the divergence theorem.

We now prove \eqref{nine2}.  For any $s\in(-1,1)$, let $J_{\tau}\Psi(x,s)$ be the Jacobian determinant of $\Psi(x,s)$, where the domain of $\Psi$ in $x$ is restricted to $x\in\redA$.  We refer to $J_{\tau}\Psi(x,s)$ as the tangential Jacobian determinant of $\Psi$.  Then \cite[Equation (2.39)]{chokski07}, says
\begin{equation}\label{nine1.1}
\frac{d}{ds}|_{s=0}J_{\tau}\Psi(x,s)=\mathrm{div}_{\tau}X(x),\qquad\forall\,x\in\redA.
\end{equation}
So, using the Chain Rule, $J_{\tau}\Psi(x,0)=1$ $\forall$ $x\in\redA$ (which follows by \eqref{nine2.3}), and \eqref{nine2.3},
\begin{flalign*}
\frac{d}{ds}|_{s=0}\int_{\redb \Omega^{(s)}}\gamma_{\sdimn}(x)dx
&=\frac{d}{ds}|_{s=0}\left[\int_{\redA}J_{\tau}\Psi(x,s)\gamma_{\sdimn}(\Psi(x,s))dx\right]\\
&=\int_{\redA}(\mathrm{div}_{\tau}X(x)-\langle X(x),x\rangle)\gamma_{\sdimn}(x)dx.
\end{flalign*}
Let $x\in\redA$.  Applying the product rule, and writing $x=\langle x,N\rangle N+(x-\langle x,N\rangle N)\equalscolon x^{N}+x^{T}$,
$$
\mathrm{div}_{\tau}(X(x)\gamma_{\sdimn}(x))
=\gamma_{\sdimn}(x)\mathrm{div}_{\tau}(X(x))+\langle\nabla \gamma_{\sdimn}(x),X(x)\rangle
=\gamma_{\sdimn}(x)\Big(\mathrm{div}_{\tau}(X(x))-\langle x^{T},X(x)\rangle \Big).
$$
Applying this equality to $X^{N}$ and $X^{T}$ separately,
$$
\mathrm{div}_{\tau}(X^{N}(x)\gamma_{\sdimn}(x))
=\gamma_{\sdimn}(x)\mathrm{div}_{\tau}(X^{N}(x)),
$$
$$
\mathrm{div}_{\tau}(X^{T}(x)\gamma_{\sdimn}(x))
=\gamma_{\sdimn}(x)\Big(\mathrm{div}_{\tau}(X^{T}(x))-\langle x,X^{T}(x)\rangle \Big).
$$
So, from the above and Remark \ref{rk10},
\begin{flalign*}
&\frac{d}{ds}|_{s=0}\int_{\redb \Omega^{(s)}}\gamma_{\sdimn}(x)dx\\
&\qquad=\int_{\redA}\mathrm{div}_{\tau}(X^{N}(x)\gamma_{\sdimn}(x))-\langle X(x),N(x)\rangle\langle N(x),x\rangle\gamma_{\sdimn}(x)+\mathrm{div}_{\tau}(X^{T}(x)\gamma_{\sdimn}(x))dx\\
&\qquad=\int_{\redA}f(x)(H(x)-\langle N(x),x\rangle)\gamma_{\sdimn}(x))+\int_{\redS}\langle X,\nu\rangle\gamma_{\sdimn}(x))dx.
\end{flalign*}
\end{proof}
\begin{remark}\label{rk23}
Let $\Omega_{1},\ldots,\Omega_{m}\subset\R^{\adimn}$ be disjoint sets with $\cup_{i=1}^{m}=\R^{\adimn}$.  Let $u_{1},\ldots,u_{m}$ denote the standard basis of $\R^{m}$.  Then we can write \eqref{nine1} in the following vector form, using $N_{ij}=-N_{ji}$ $\forall$ $1\leq i<j\leq m$,
\begin{flalign*}
\frac{d}{ds}|_{s=0}\sum_{i=1}^{m}\gamma_{\adimn}(\Omega_{i}^{(s)})u_{i}
&=\sum_{i=1}^{m}u_{i}\sum_{j\neq i}\int_{\Sigma_{ij}}\langle X, N_{ij}\rangle\gamma_{\adimn}(x)dx\\
&=\sum_{1\leq i<j\leq m}(u_{i}-u_{j})\int_{\Sigma_{ij}}\langle X, N_{ij}\rangle\gamma_{\adimn}(x)dx\\
&=\frac{1}{\sqrt{2\pi}}\sum_{1\leq i<j\leq m}(u_{i}-u_{j})\frac{\int_{\Sigma_{ij}}\langle X, N_{ij}\rangle\gamma_{\sdimn}(x)dx}{\gamma_{\sdimn}(\Sigma_{ij})}\gamma_{\sdimn}(\Sigma_{ij}).
\end{flalign*}

\end{remark}

\begin{definition}\label{defnote}
Below, for any $1\leq i<j\leq m$, we denote $\Sigma_{ij}\colonequals(\redb\Omega_{i})\cap(\redb\Omega_{j})$, and
$$C\colonequals\bigcup_{1\leq i<j<k\leq m}(\redS_{ij})\cap (\redS_{jk})\cap (\redS_{ki}).$$
We also let $X\in C_{0}^{\infty}(\R^{\adimn},\R^{\adimn})$ and for any $1\leq i<j\leq m$, we denote $f_{ij}(x)\colonequals\langle X(x),N_{ij}(x)\rangle$ for all $x\in\Sigma_{ij}$.  Recall that $N_{ij}$ is the unit normal vector pointing from $\Omega_{i}$ into $\Omega_{j}$, and $H_{ij}\colonequals\mathrm{div}(N_{ij})$ is the mean curvature of $N_{ij}$.  And $\nu_{ij}$ is the unit normal to $\redb\Sigma_{ij}$ pointing exterior to $\Sigma_{ij}$.
\end{definition}

\begin{figure}[h]\label{notationfig}
\centering
\def\svgwidth{.4\textwidth}
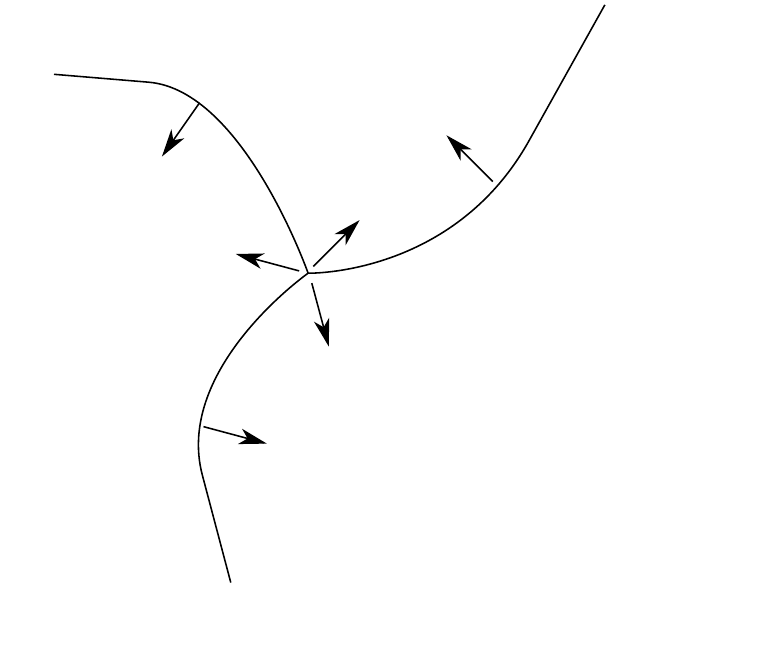
\caption{Notation for sets and normal vectors.}
\end{figure}

The following Lemma can be compared with \cite[Lemma 3.1]{hutchings02}.

\begin{lemma}[\embolden{First Variation for Minimizers}]\label{lemma24}
Suppose $\Omega_{1},\ldots,\Omega_{m}$ minimize Problem \ref{prob1}.  Then $\forall$ $1\leq i<j\leq m$, $\exists$ $\lambda_{ij}\in\R$ such that
\begin{equation}\label{nine5.3}
\left.
\begin{aligned}
H_{ij}(x)&=\langle x,N_{ij}(x)\rangle+\lambda_{ij},\qquad\forall\,x\in\Sigma_{ij},\quad\forall\,1\leq i<j\leq m\\
0&=\lambda_{ij}+\lambda_{jk}+\lambda_{ki},\qquad\forall\,1\leq i<j<k\leq m\,\,\mathrm{such}\,\,\mathrm{that}\,\,\Sigma_{ij}\cap\Sigma_{jk}\cap\Sigma_{ki}\neq\emptyset\\
0&=\nu_{ij}+\nu_{jk}+\nu_{ki},\qquad\forall\,1\leq i<j<k\leq m\,\,\mathrm{such}\,\,\mathrm{that}\,\,\Sigma_{ij}\cap\Sigma_{jk}\cap\Sigma_{ki}\neq\emptyset.
\end{aligned}
\right\}
\end{equation}
\end{lemma}
\begin{proof}
From Lemma \ref{lemma52.6}, Assumption \ref{as1} holds.  From \eqref{nine2}, if $X\in C_{0}^{\infty}(\R^{\adimn},\R^{\adimn})$
\begin{flalign*}
&\frac{d}{ds}|_{s=0}\sum_{1\leq i<j\leq m}\int_{\Sigma_{ij}^{(s)}}\gamma_{\sdimn}(x)dx\\
&\qquad\qquad=\sum_{1\leq i<j\leq m}\int_{\Sigma_{ij}}(H_{ij}-\langle N_{ij},x\rangle)f_{ij}\gamma_{\sdimn}(x)dx
+\int_{C}\langle X,\sum_{1\leq i<j\leq m}\nu_{ij}\rangle \gamma_{\sdimn}(x)dx.
\end{flalign*}
Choosing $X$ to be supported in the neighborhood of a given $x\in\Sigma_{ij}$ implies the first part of \eqref{nine5.3}.  The final part of \eqref{nine5.3} follows by Lemma \ref{lemma52.6}, i.e. Assumption \ref{as1}(ii).  Choosing $X$ to be supported in the neighborhood of a given $x\in\Sigma_{ij}\cap\Sigma_{jk}\cap\Sigma_{ki}$ implies the second part of \eqref{nine5.3}.
\end{proof}
\begin{remark}[{\cite[Theorem 4.9(ii)]{milman18a}}]\label{rk27}
Let $\Omega_{1},\ldots,\Omega_{m}$ minimize Problem \ref{prob1}.  The middle condition of \eqref{nine5.3} is equivalent to the existence of $\lambda\colonequals(\lambda_{1},\ldots,\lambda_{m})\in\R^{m}$ such that $\lambda_{ij}=\lambda_{i}-\lambda_{j}$ for all $1\leq i<j\leq m$, and with $\sum_{i=1}^{m}\lambda_{i}=0$.  Then, combining Lemmas \ref{lemma10}, \ref{lemma24}, Assumption \ref{as1}, and $N_{ij}=-N_{ji}$ for all $1\leq i<j\leq m$,
\begin{flalign*}
&\frac{d}{ds}|_{s=0}\sum_{1\leq i<j\leq m}\int_{\Sigma_{ij}^{(s)}}\gamma_{\sdimn}(x)dx\\
&\stackrel{\eqref{nine2}}{=}\sum_{1\leq i<j\leq m}\int_{\Sigma_{ij}}(H_{ij}-\langle x,N_{ij}\rangle)\langle X,N_{ij}\rangle\gamma_{\sdimn}(x)dx
=\sum_{1\leq i<j\leq m}\lambda_{ij}\int_{\Sigma_{ij}}\langle X,N_{ij}\rangle\gamma_{\sdimn}(x)dx\\
&=\sum_{1\leq i<j\leq m}(\lambda_{i}-\lambda_{j})\int_{\Sigma_{ij}}\langle X,N_{ij}\rangle\gamma_{\sdimn}(x)dx
=\sum_{1\leq i,j\leq m\colon i\neq j}\lambda_{i}\int_{\Sigma_{ij}}\langle X,N_{ij}\rangle\gamma_{\sdimn}(x)dx\\
&=\sum_{i=1}^{m}\lambda_{i}\sum_{1\leq j\leq m\colon j\neq i}\int_{\Sigma_{ij}}\langle X,N_{ij}\rangle\gamma_{\sdimn}(x)dx
\stackrel{\eqref{nine1}}{=}\sum_{i=1}^{m}\lambda_{i}\sqrt{2\pi}\frac{d}{ds}|_{s=0}\gamma_{\adimn}(\Omega_{i})
=\sqrt{2\pi}\langle\lambda,V'(0)\rangle,
\end{flalign*}
where $V(s)\colonequals(\gamma_{\adimn}(\Omega_{1}^{(s)}),\ldots,\gamma_{\adimn}(\Omega_{m}^{(s)}))$ for all $s\in(-1,1)$.  Similarly, by taking another derivative in $s$, we get
$$
\sum_{1\leq i<j\leq m}\lambda_{ij}\frac{d^{2}}{ds^{2}}|_{s=0}\int_{\Sigma_{ij}^{(s)}}\langle X,N_{ij}\rangle\gamma_{\sdimn}(x)dx
=\sqrt{2\pi}\langle\lambda,V''(0)\rangle.
$$
%\begin{flalign*}
%&\sum_{1\leq i<j\leq m}\lambda_{ij}\frac{d}{ds}|_{s=0}\int_{\Sigma_{ij}^{(s)}}\langle X,N_{ij}\rangle\gamma_{\sdimn}(x)dx\\
%&=\sum_{1\leq i<j\leq m}(\lambda_{i}-\lambda_{j})\frac{d}{ds}|_{s=0}\int_{\Sigma_{ij}^{(s)}}\langle X,N_{ij}\rangle\gamma_{\sdimn}(x)dx
%=\sum_{1\leq i,j\leq m\colon i\neq j}\lambda_{i}\frac{d}{ds}|_{s=0}\int_{\Sigma_{ij}^{(s)}}\langle X,N_{ij}\rangle\gamma_{\sdimn}(x)dx\\
%&=\sum_{i=1}^{m}\lambda_{i}\sum_{1\leq j\leq m\colon j\neq i}\frac{d}{ds}|_{s=0}\int_{\Sigma_{ij}^{(s)}}\langle X,N_{ij}\rangle\gamma_{\sdimn}(x)dx
%\stackrel{\eqref{nine1}}{=}\sum_{i=1}^{m}\lambda_{i}\sqrt{2\pi}\frac{d^{2}}{ds^{2}}|_{s=0}\gamma_{\adimn}(\Omega_{i})
%=\sqrt{2\pi}\langle\lambda,V''(0)\rangle,
%\end{flalign*}
\end{remark}

For any $1\leq i<j\leq m$, let $f\colon C_{0}^{\infty}(\Sigma_{ij})\to\R$ and define

\begin{equation}\label{one3.5n}
L_{ij}f(x)\colonequals \Delta f(x)-\langle x,\nabla f(x)\rangle+\vnormt{A_{x}}^{2}f(x)+f(x),\qquad\forall\,x\in\Sigma_{ij}.
\end{equation}
\begin{equation}\label{one3.5o}
\mathcal{L}_{ij}f(x)\colonequals \Delta f(x)-\langle x,\nabla f(x)\rangle,\qquad\forall\,x\in\Sigma_{ij}.
\end{equation}

The Lemma below can be compared with the corresponding \cite[Proposition 3.3]{hutchings02}.

\begin{lemma}[\embolden{Second Variation of Gaussian Surface Area for Minimizers}]\label{lemma21}
Let $X\in C_{0}^{\infty}(\R^{\adimn},\R^{\adimn})$.  Suppose $\Omega_{1},\ldots,\Omega_{m}$ minimize Problem \ref{prob1}.  Then
\begin{equation}\label{nine5}
\begin{aligned}
&\frac{d^{2}}{ds^{2}}|_{s=0}\sum_{1\leq i<j\leq m}\int_{\Sigma_{ij}^{(s)}}\gamma_{\sdimn}(x)dx
=\sum_{1\leq i<j\leq m}\Big(-\int_{\Sigma_{ij}}f_{ij}L_{ij}f_{ij}\gamma_{\sdimn}(x)dx\\
&\qquad\qquad\qquad\qquad\qquad+\lambda_{ij}\frac{d}{ds}|_{s=0}\int_{\Sigma_{ij}^{(s)}}f_{ij}(x)\gamma_{\sdimn}(x)dx
+\frac{d}{ds}|_{s=0}\int_{\redS_{ij}^{(s)}}\langle X,\nu_{ij}\rangle\Big)\gamma_{\sdimn}(x)dx.
\end{aligned}
\end{equation}
\end{lemma}
\begin{proof}
Let $\Sigma$ be an $n$-dimensional $C^{\infty}$ hypersurface with boundary.  We let $'$ denote $\frac{\partial}{\partial s}|_{s=0}$.  From Lemma \ref{lemma10} we have
\begin{flalign*}
&\frac{d^{2}}{ds^{2}}|_{s=0}\int_{\Sigma^{(s)}}\gamma_{\sdimn}(x)dx
=\int_{\Sigma}(H(x)-\langle N(x),x\rangle)'f(x)\gamma_{\sdimn}(x)dx\\
&\qquad+\int_{\Sigma}(H(x)-\langle N(x),x\rangle)[f(x)\gamma_{\sdimn}(x)dx]'
+\frac{d}{ds}|_{s=0}\int_{\redS^{(s)}}\langle X,\nu\rangle\gamma_{\sdimn}(x)dx.
\end{flalign*}
From \eqref{nine2.3}, $x'=X=X^{N}+X^{T}=fN+X^{T}$.  Also, $H'=-\Delta f-\vnormt{A}^{2}f$, $N'=-\nabla f$, \cite[A.3, A.4]{colding12a} (the latter calculations require writing $\Sigma^{(s)}$ in the form $\{x+ sN(x)+O_{x}(s^{2})\colon x\in\Sigma\}$).  So,
$$(H-\langle N,x\rangle)'=-\Delta f-\vnormt{A}^{2}f-\langle N,fN+X^{T}\rangle-\langle x,\nabla f\rangle\stackrel{\eqref{three4.5}}{=}-Lf.$$
In summary,
\begin{flalign*}
&\frac{d^{2}}{ds^{2}}|_{s=0}\int_{\Sigma^{(s)}}\gamma_{\sdimn}(x)dx
=-\int_{\Sigma}fLf\gamma_{\sdimn}(x)dx\\
&\qquad\qquad\qquad+\int_{\Sigma}(H(x)-\langle N(x),x\rangle)[f(x)\gamma_{\sdimn}(x)dx]'
+\frac{d}{ds}|_{s=0}\int_{\redS^{(s)}}\langle X,\nu\rangle\gamma_{\sdimn}(x)dx.
\end{flalign*}
Summing over all $1\leq i<j\leq m$, and applying \eqref{nine5.3},
\begin{flalign*}
&\frac{d^{2}}{ds^{2}}|_{s=0}\sum_{1\leq i<j\leq m}\int_{\Sigma_{ij}^{(s)}}\gamma_{\sdimn}(x)dx
=\sum_{1\leq i<j\leq m}-\int_{\Sigma_{ij}}f_{ij}L_{ij}f_{ij}\gamma_{\sdimn}(x)dx\\
&\qquad\qquad\qquad\qquad+\lambda_{ij}\int_{\Sigma_{ij}}[f_{ij}(x)\gamma_{\sdimn}(x)dx]'
+\frac{d}{ds}|_{s=0}\int_{\redS_{ij}^{(s)}}\langle X,\nu_{ij}\rangle\gamma_{\sdimn}(x)dx.
\end{flalign*}
\end{proof}

Below, we need the following combinatorial Lemma, the case $m=3$ being treated in \cite[Proposition 3.3]{hutchings02}.
\begin{lemma}\label{lemma25.3}
Let $m\geq3$.  Let
$$D_{1}\colonequals \{(x_{ij})_{1\leq i\neq j\leq m}\in\R^{\binom{m}{2}}\colon \forall\,1\leq i\neq j\leq m,\quad x_{ij}=-x_{ji},\,\sum_{j\in\{1,\ldots,m\}\colon j\neq i}x_{ij}=0\}.$$
\begin{flalign*}
D_{2}\colonequals \{(x_{ij})_{1\leq i\neq j\leq m}\in\R^{\binom{m}{2}}
&\colon  \forall\,1\leq i\neq j\leq m,\quad x_{ij}=-x_{ji},\\
&\forall\,1\leq i<j<k\leq m\quad x_{ij}+x_{jk}+x_{ki}=0\}.
\end{flalign*}
Let $x\in D_{1}$ and let $y\in D_{2}$.  Then $\sum_{1\leq i<j\leq m}x_{ij}y_{ij}=0$.
\end{lemma}
\begin{proof}
$D_{1}$ is defined to be perpendicular to vectors in $D_{2}$, and vice versa.  That is, $D_{1}$ and $D_{2}$ are orthogonal complements of each other, and in terms of vector spaces, $D_{1}\oplus D_{2}=\R^{\binom{m}{2}}$.  Consequently, the inner product of any $x\in D_{1}$ and $y\in D_{2}$ is zero.
\end{proof}

It is well-known that compactly supported variations such that $\frac{d}{ds}|_{s=0}\gamma_{\adimn}(\Omega_{i}^{(s)})=0$ for all $1\leq i\leq m$ can be modified such that $\frac{d^{k}}{ds^{k}}|_{s=0}\gamma_{\adimn}(\Omega_{i}^{(s)})=0$ for all $1\leq i\leq m$ and for all $k\geq1$ while preserving the second variation.  Such an application of the implicit function theorem appears e.g. in \cite[Proposition 3.3]{hutchings02} or \cite[Lemma 2.4]{barbosa84}.  This argument can be extended to noncompact variations \cite[Lemma 1]{barchiesi16}.

The Lemma below can be compared with the corresponding result \cite[Lemma 3.2]{hutchings02} and \cite[Lemmas 4.12 and 5.2]{milman18a}.

\begin{lemma}[\embolden{Extension Lemma for Existence of Volume-Preserving Variations}]\label{lemma27}
For any $1\leq i<j\leq m$, let $f_{ij}\in C_{0}^{\infty}(\Sigma_{ij})$ satisfy
\begin{equation}\label{eight1}
\forall\,1\leq i<j<k\leq m,\,\forall\,x\in\Sigma_{ij}\cap \Sigma_{jk}\cap\Sigma_{ki},\quad f_{ij}(x)+f_{jk}(x)+f_{ki}(x)=0.
\end{equation}
Then there exists a vector field $X\in C_{0}^{\infty}(\R^{\adimn},\R^{\adimn})$ such that
\begin{equation}\label{eight3}
\forall\,1\leq i<j\leq m,\quad\forall\, x\in\Sigma_{ij},\,\langle X(x),N_{ij}(x)\rangle=f_{ij}(x).
\end{equation}
If additionally
\begin{equation}\label{eight2}
\forall\,1\leq i\leq m,\quad \sum_{j\in\{1,\ldots,m\}\colon j\neq i}\int_{\Sigma_{ij}}f_{ij}\gamma_{\sdimn}(x)dx=0,
\end{equation}
then $X$ can also be chosen to be volume preserving:
$$\forall\,1\leq i\leq m,\quad\forall\,s\in(-\epsilon,\epsilon),\quad \gamma_{\adimn}(\Omega_{i}^{(s)})=\gamma_{\adimn}(\Omega_{i}).$$
\end{lemma}
\begin{proof}
By assumption, $\exists$ a vector field $Z\colon\R^{\adimn}\to\R^{\adimn}$ such that
$$\langle Z(x),N_{ij}(x)\rangle=f_{ij}(x),\qquad\forall\, x\in C,\quad\forall\,1\leq i<j\leq m.$$
Then $Z$ can be extended to all of $\cup_{1\leq i<j\leq m}\Sigma_{ij}$ by e.g. Whitney Extension.  Let $I$ be a subset of $\{(i,j)\colon 1\leq i<j\leq m\}$ of size $m-1$.  For all $(i,j)\in I$, let $g_{ij}\colon\Sigma_{ij}\to\R$ be compactly supported, nonnegative, $C^{\infty}$ functions and let $\widetilde{g}_{ij}$ be any smooth extension of $g_{ij}$ to $\R^{\adimn}$ that is supported in a neighborhood of the interior $\Sigma_{ij}$, disjoint from all $\Sigma_{i'j'}$ with $(i',j')\neq(i,j)$.  Similarly, let $\widetilde{N}_{ij}$ be any smooth extension of $N_{ij}$ to $\R^{\adimn}$.  Consider the map $\widetilde{\Psi}\colon\R^{\adimn}\times(-1,1)\times(-1,1)^{m-1}\to\R^{\adimn}$ defined by
$$\widetilde{\Psi}(x,s,\{t_{ij}\}_{(i,j)\in I})\colonequals x+sZ+\sum_{(i,j)\in I}t_{ij}\widetilde{g}_{ij}\widetilde{N}_{ij},
\qquad\forall\,s\in(-1,1),\,\{t_{ij}\}_{(i,j)\in I}\in(-1,1)^{m-1}.$$
And consider the vector-valued function
$$V(s,\{t_{ij}\}_{(i,j)\in I})\colonequals\Big(\gamma_{\adimn}(\Omega_{1}^{(s,\{t_{ij}\}_{(i,j)\in I})}),\ldots,\gamma_{\adimn}(\Omega_{m}^{(s,\{t_{ij}\}_{(i,j)\in I})})\Big).$$
Then $V\colon\R^{m}\to\R^{m}$, and the image of $V$ is at most $(m-1)$-dimensional, since the sum of the entries of $V$ is equal to $1$.  Consider the equation $V=\mathrm{constant}$.  Then the Jacobian of $V$ has maximal rank.  So, by the Implicit Function Theorem, for every $(i,j)\in I$, there exists a function $t_{ij}\colon(-1,1)\to\R$ such that $V(s,\{t_{ij}(s)\}_{(i,j)\in I})=\mathrm{constant}$ for all $s\in(-1,1)$.  Since the Jacobian of $V$ has maximal rank and \eqref{eight2} holds, it follows from the chain rule that $t_{ij}'(0)=0$ for all $(i,j)\in I$.  So, if we let $X$ be the vector field for $\Psi(x,s)\colonequals\widetilde{\Psi}(x,s,\{t_{ij}(s)\}_{(i,j)\in I})$ satisfying \eqref{nine2.3}.  Then \eqref{eight3} holds for $X$.
\end{proof}

\begin{lemma}[\embolden{Volume-Preserving Second Variation of Gaussian Surface Area for Minimizers}]\label{lemma28}
Let $\Omega_{1},\ldots,\Omega_{m}$ minimize Problem \ref{prob1}.  $\forall$ $1\leq i<j\leq m$, let $f_{ij}\in C_{0}^{\infty}(\Sigma_{ij})$ satisfy \eqref{eight1} and \eqref{eight2}.  Let $X$ be the vector field guaranteed to exist from Lemma \ref{lemma27}.  Then
\begin{flalign*}
&\frac{d^{2}}{ds^{2}}|_{s=0}\sum_{1\leq i<j\leq m}\int_{\Sigma_{ij}^{(s)}}\gamma_{\sdimn}(x)dx\\
&\qquad\qquad\qquad=\sum_{1\leq i<j\leq m}-\int_{\Sigma_{ij}}f_{ij}L_{ij}f_{ij}\gamma_{\sdimn}(x)dx
+\frac{d}{ds}|_{s=0}\int_{\redS_{ij}^{(s)}}\langle X,\nu_{ij}\rangle\gamma_{\sdimn}(x)dx.
\end{flalign*}
\end{lemma}
\begin{proof}
Assumption \ref{as1} holds by Lemma \ref{lemma52.6}.  From Lemma \ref{lemma24}, \eqref{nine5.3} holds.  Since the volumes are preserved, for any $1\leq i\leq m$, we have $\sum_{j\neq i}\frac{d}{ds}|_{s=0}\int_{\Sigma_{ij}^{(s)}}f_{ij}(x)\gamma_{\sdimn}(x)dx=0$.  Combining Lemmas \ref{lemma21} and \ref{lemma25.3} shows that the middle term from Lemma \ref{lemma21} vanishes.
\end{proof}

\begin{remark}\label{rk5}
$\forall$ $1\leq i<j<k\leq m$, and $\forall$ $x\in(\redS_{ij})\cap (\redS_{jk})\cap (\redS_{ki})$, define
$$q_{ij}(x)\colonequals[\langle \nabla_{\nu_{kj}}\nu_{kj},N_{kj}\rangle+\langle \nabla_{\nu_{ki}}\nu_{ki},N_{ki}\rangle]/\sqrt{3}.$$
Note that $q_{ij}+q_{jk}+q_{ki}=0$ since $N_{ij}=-N_{ji}$ by Definition \ref{defnote} and $q_{ij}=q_{ji}$.

Compared to \cite{hutchings02}, note that we have the opposite sign convention for the second fundamental form and for $\nu_{ij}$.
\end{remark}
\begin{lemma}[{\cite[Lemma 3.6]{hutchings02}}]\label{lemma29.9}
 $\forall$ $1\leq i<j\leq m$, let $f_{ij}\in C_{0}^{\infty}(\Sigma_{ij})$ satisfy \eqref{eight1}.  Let $X$ be the vector field guaranteed to exist from Lemma \ref{lemma27}.  Then
\begin{flalign*}
&\frac{d}{ds}|_{s=0}\sum_{1\leq i<j\leq m}\int_{\Sigma_{ij}^{(s)}}\langle X,\nu_{ij}\rangle\gamma_{\sdimn}(x)dx\\
&=\sum_{1\leq i<j<k\leq m}\int_{\redS_{ij}\cap\redS_{jk}\cap \redS_{ki}}\langle X,\nabla_{X}(\nu_{ij}+\nu_{jk}+\nu_{ki})\rangle\gamma_{\sdimn}(x)dx\\
 &=\sum_{1\leq i<j<k\leq m}\int_{\redS_{ij}\cap\redS_{jk}\cap \redS_{ki}}
\Big([\nabla_{\nu_{ij}}f_{ij}+q_{ij}f_{ij}]f_{ij}+[\nabla_{\nu_{jk}}f_{jk}+q_{jk}f_{jk}]f_{jk}\\
 &\qquad\qquad\qquad\qquad\qquad\qquad+[\nabla_{\nu_{ki}}f_{ki}+q_{ki}f_{ki}]f_{ki}\Big)\gamma_{\sdimn}(x)dx.
 \end{flalign*}
\end{lemma}
For a proof of Lemma \ref{lemma29.9}, see \cite[Lemma 3.6]{hutchings02}.  Applying the above Lemma, we get

\begin{lemma}\label{lemma36.5}
If for all $1\leq i<j<k\leq m$ we have $f_{ij}\in C_{0}^{\infty}(\Sigma_{ij})$ satisfying \eqref{eight1} and
$$\nabla_{\nu_{ij}}f_{ij}+q_{ij}f_{ij}=\nabla_{\nu_{jk}}f_{jk}+q_{jk}f_{jk}=\nabla_{\nu_{ij}}f_{ij}+q_{ij}f_{ij},\qquad\forall\, 1\leq i<j<k\leq m,$$
then
$$\frac{d}{ds}|_{s=0}\sum_{1\leq i<j\leq m}\int_{\Sigma_{ij}^{(s)}}\langle X,\nu_{ij}\rangle\gamma_{\sdimn}(x)dx=0.$$
\end{lemma}

\section{Second Variation as a Quadratic Form}\label{seceigs}

\begin{definition}[\embolden{Admissible Functions}]\label{def1}
Define $\mathcal{F}$ be the set of functions $(f_{ij})_{1\leq i<j\leq m}$ such that
\begin{itemize}
\item  $\forall$ $1\leq i<j\leq m$, $f_{ij}\colon\Sigma_{ij}\to\R$, $\int_{\Sigma_{ij}}f_{ij}^{2}\gamma_{\sdimn}(x)dx<\infty$ and $\int_{\Sigma_{ij}}\vnormt{\nabla f_{ij}}^{2}\gamma_{\sdimn}(x)dx<\infty$.
\item$\forall\,1\leq i<j<k\leq m,\,\forall\,x\in\Sigma_{ij}\cap \Sigma_{jk}\cap\Sigma_{ki},\, f_{ij}(x)+f_{jk}(x)+f_{ki}(x)=0$.
%\item$\nabla_{\nu_{ij}}f_{ij}+q_{ij}f_{ij}=\nabla_{\nu_{jk}}f_{jk}+q_{jk}f_{jk}=\nabla_{\nu_{ij}}f_{ij}+q_{ij}f_{ij}$.
%\item $\exists 1\leq i'<j'\leq m\colon f_{i'j'}\neq0$.
\end{itemize}
The second condition is well-defined by e.g. a (local) Sobolev Trace inequality \cite{feo13}.
\end{definition}

\begin{definition}[\embolden{Quadratic Form Associated to Second Variations}]\label{def2}
For any $F=(f_{ij})_{1\leq i<j\leq m},G=(g_{ij})_{1\leq i<j\leq m}\in\mathcal{F}$, define the following quantities if they exist:
\begin{equation}\label{quaddef}
\begin{aligned}
&Q(F,G)\colonequals
\sum_{1\leq i<j\leq m}-\int_{\Sigma_{ij}}g_{ij}L_{ij}f_{ij}\gamma_{\sdimn}(x)dx+\sum_{1\leq i<j<k\leq m}\int_{\redS_{ij}\cap\redS_{jk}\cap \redS_{ki}}\\
&\qquad\Big([\nabla_{\nu_{ij}}f_{ij}+q_{ij}f_{ij}]g_{ij}+[\nabla_{\nu_{jk}}f_{jk}+q_{jk}f_{jk}]g_{jk}
+[\nabla_{\nu_{ki}}f_{ki}+q_{ki}f_{ki}]g_{ki}\Big)\gamma_{\sdimn}(x).
 \end{aligned}
 \end{equation}
 \begin{equation}\label{ipdef}
 \langle F,G\rangle\colonequals\sum_{1\leq i<j\leq m}\int_{\Sigma_{ij}}f_{ij}g_{ij}\gamma_{\sdimn}(x)dx.
 \end{equation}
 \end{definition}

Using \eqref{one3.5n}, define $\L\colon \cup_{1\leq i<j\leq m}C_{0}^{\infty}(\Sigma_{ij})\to \cup_{1\leq i<j\leq m}C_{0}^{\infty}(\Sigma_{ij})$ by
\begin{equation}\label{seven00}
\L((f_{ij})_{1\leq i<j\leq m})\colonequals (L_{ij}f_{ij})_{1\leq i<j\leq m}.
\end{equation}
Using also \eqref{one3.5o}
\begin{equation}\label{seven00p}
\mL((f_{ij})_{1\leq i<j\leq m})\colonequals (\mathcal{L}_{ij}f_{ij})_{1\leq i<j\leq m}.
\end{equation}

\begin{lemma}[\embolden{Integration by Parts}]\label{lemma32.5}
Let $F,G\in\mathcal{F}\cap C_{0}^{\infty}(\cup_{1\leq i<j\leq m}\Sigma_{ij})$.  Then
\begin{flalign*}
&Q(F,G)\colonequals
\sum_{1\leq i<j\leq m}\int_{\Sigma_{ij}}[\langle\nabla f_{ij},\nabla g_{ij}\rangle -f_{ij}g_{ij}(\vnormt{A}^{2}+1)]\gamma_{\sdimn}(x)dx\\
&\qquad+\sum_{1\leq i<j<k\leq m}\int_{\redS_{ij}\cap\redS_{jk}\cap \redS_{ki}}
 [q_{ij}f_{ij}g_{ij}+q_{jk}f_{jk}g_{jk}+q_{ki}f_{ki}g_{ki}]\gamma_{\sdimn}(x)dx.
 \end{flalign*}
 In particular, $Q(F,G)=Q(G,F)$, so that $Q$ is symmetric.
\end{lemma}
\begin{proof}
From the divergence theorem for an $\sdimn$-dimensional $C^{\infty}$ orientable hypersurface $\Sigma$ with $C^{\infty}$ boundary, if $f,g\colon\Sigma\to\R$, then
\begin{flalign*}
&\int_{\Sigma}(\mathcal{L}f)g\gamma_{\sdimn}(x)dx
\stackrel{\eqref{three4.3}}{=}\int_{\Sigma}(\Delta f-\langle x,\nabla f\rangle)g\gamma_{\sdimn}(x)dx
=\int_{\Sigma}\mathrm{div}_{\tau}(\gamma_{\sdimn}(x)\nabla f)gdx\\
&=\int_{\Sigma}\Big([\mathrm{div}_{\tau}(g\gamma_{\sdimn}(x)\nabla f)]-\langle\nabla f,\nabla g\rangle\gamma_{\sdimn}(x)\Big)dx
=\int_{\partial \Sigma}\langle\nabla f,\nu\rangle g\gamma_{\sdimn}(x)-\int_{\Sigma}\langle\nabla f,\nabla g\rangle\gamma_{\sdimn}(x)dx.
\end{flalign*}
As usual, $\nu$ denotes the exterior pointing unit normal to $\partial\Sigma$.  Substituting into the definition of $Q(F,G)$ and using \eqref{three4.3} and \eqref{three4.5} completes the proof.
\end{proof}

\begin{lemma}[{\cite[Lemma 1]{barchiesi16}}]\label{lemma60}
Let $\Omega_{1},\ldots,\Omega_{m}$ satisfy Assumption \ref{as1}.  Then there exists a sequence of $C^{\infty}$ functions $\eta_{1}\leq \eta_{2}\leq\cdots\colon\cup_{i=1}^{m}\redb\Omega_{i}\to[0,1]$ supported in $M_{\sdimn}\cup M_{\sdimn-1}\cup M_{\sdimn-2}$ (using the notation of Assumption \ref{as1}) such that
$$\forall\,x\in \cup_{i,j=1}^{m}\redb\Sigma_{ij},\quad\lim_{u\to\infty}\eta_{u}(x)=1,$$
$$\lim_{u\to\infty}\sum_{1\leq i<j\leq m}\int_{\Sigma_{ij}}[(1-\eta_{u})^{2}+\vnorm{\nabla(1-\eta_{u})}^{2}]\gamma_{\sdimn}(x)dx=0.$$
\end{lemma}
\begin{proof}
By Assumption \ref{as1}, $S\colonequals\cup_{i=1}^{m}\partial\Omega_{i}\setminus(M_{\sdimn}\cup M_{\sdimn-1}\cup M_{\sdimn-2})$ has zero $(\sdimn-2)$-dimensional Hausdorff measure, so the assertion follows e.g. by \cite[Lemma 1]{barchiesi16}.  More specifically, fix $\epsilon,r>0$, and note that there exist $x^{(1)},\ldots,x^{(m)}\in\R^{\adimn}$ and $0<r_{1},\ldots,r_{m}<1/2$ such that
$$\{x\in S\colon\vnorm{x}\leq r\}\subset\cup_{i=1}^{m}B(x_{i},r_{i}),\qquad \sum_{i=1}^{m}r_{i}^{\sdimn-2}<\epsilon.$$
Then, for each $1\leq i\leq m$, let $0\leq \theta_{i}\leq 1$ be a smooth function such that $\theta_{i}=0$ inside $B(x_{i},2r_{i})$, such that $\theta_{i}=1$ outside $B(x_{i},3r_{i})$, and such that $\vnorm{\nabla\theta_{i}}\leq 2/r_{i}$ in $B(x_{i},3r_{i})\setminus B(x_{i},2r_{i})$.  Let $0\leq\theta_{0}\leq 1$ be a smooth function such that $\theta_{0}=1$ inside $B(0,r-3)$ and $\theta_{0}=0$ outside $B(0,r-2)$, and then define $\widetilde{\eta}_{r,\epsilon}\colonequals\min_{0\leq i\leq m}\theta_{i}$.  Then
\begin{flalign*}
&\sum_{1\leq i<j\leq m}\int_{\Sigma_{ij}}\vnormf{\nabla(1-\widetilde{\eta}_{r,\epsilon})}^{2}\gamma_{\sdimn}(x)dx\\
&\qquad\leq\int_{S\cap B(0,r-2)\setminus B(0,r-3)}4\gamma_{\sdimn}(x)dx+\sum_{k=1}^{m}\int_{S\cap B(x_{k},3r_{k})\setminus B(x_{k},2r_{k})}\frac{4}{r_{k}^{2}}\gamma_{\sdimn}(x)dx\\
&\qquad\leq4m\Big(r^{\sdimn}e^{-(r-3)^{2}/2}+3^{\sdimn}\sum_{k=1}^{m}r_{k}^{\sdimn-2}\Big)
\leq4m\Big(r^{\sdimn}e^{-(r-3)^{2}/2}+3^{\sdimn}\epsilon\Big).
\end{flalign*}
The penultimate inequality used Lemma \ref{lemma28.8}.  Finally, let $\epsilon\colonequals\epsilon(r)\to0$ as $r\to\infty$, and define $\eta_{r}\colonequals \eta_{r,\epsilon(r)}$.  The bound on the integral of $(1-\eta_{u})^{2}$ follows similarly from Lemma \ref{lemma28.8}, since
\begin{flalign*}
&\sum_{1\leq i<j\leq m}\int_{\Sigma_{ij}}(1-\widetilde{\eta}_{r,\epsilon})^{2}\gamma_{\sdimn}(x)dx\\
&\qquad\leq\int_{S\cap B(0,r-2)\setminus B(0,r-3)}4\gamma_{\sdimn}(x)dx+\sum_{k=1}^{m}\int_{S\cap B(x_{k},3r_{k})\setminus B(x_{k},2r_{k})}4\gamma_{\sdimn}(x)dx\\
&\qquad\leq4m\Big(r^{\sdimn}e^{-(r-3)^{2}/2}+3^{\sdimn}\sum_{k=1}^{m}r_{k}^{\sdimn}\Big)
\leq4m\Big(r^{\sdimn}e^{-(r-3)^{2}/2}+3^{\sdimn}\epsilon\Big).
\end{flalign*}
\end{proof}

\begin{lemma}[\embolden{Non-Compact Variations}]\label{lemma97}
Let $\Omega_{1},\ldots,\Omega_{m}$ satisfy Assumption \ref{as1}.  Let $F,G\in\mathcal{F}$.  Assume that
$$\sum_{1\leq i<j\leq m}\int_{\Sigma_{ij}}\abs{L_{ij}f_{ij}}^{2}\gamma_{\sdimn}(x)dx<\infty.$$
Assume $\forall\,1\leq i<j<k\leq m$, $\forall$ $x\in(\redS_{ij})\cap(\redS_{jk})\cap(\redS_{ki})$, the following holds at $x$.
\begin{equation}\label{four75}
\nabla_{\nu_{ij}}f_{ij}+q_{ij}f_{ij}=\nabla_{\nu_{jk}}f_{jk}+q_{jk}f_{jk}=\nabla_{\nu_{ki}}f_{ki}+q_{ki}f_{ki}.
\end{equation}
Then $Q(F,F)$ and $Q(F,G)$ are well-defined real numbers.  Moreover,
$$Q(F,F)=-\langle LF,F\rangle,\qquad Q(F,G)=-\langle LF,G\rangle.$$
Also, $\exists$ a sequence $\phi_{1},\phi_{2},\ldots\in C_{0}^{\infty}(\Sigma)$ with  $0\leq\phi_{1}\leq\phi_{2}\leq\cdots \leq1$ on $\R^{\adimn}$ converging pointwise to $1$ such that
$$\lim_{u\to\infty}Q(\phi_{u}F,G)=\lim_{u\to\infty}Q(\phi_{u}F,\phi_{u}G)=Q(F,G).$$
\end{lemma}
\begin{proof}
Let  $\Sigma\colonequals\cup_{1\leq i<j\leq m}\Sigma_{ij}$.  Let $\phi\in C_{0}^{\infty}(\Sigma)$ with $0\leq\phi\leq 1$, $\phi=1$ when $\vnorm{x}\leq r$, $\phi=0$ when $\vnorm{x}>r+2$ and $\vnorm{\nabla\phi}\leq1$ on $\Sigma$.  From Lemma \ref{lemma32.5},
$$Q(\phi F,G)-Q(F,\phi G)=\sum_{1\leq i<j\leq m}\int_{\Sigma_{ij}}\Big(f_{ij}\langle\nabla \phi,\nabla g_{ij}\rangle-g_{ij}\langle\nabla \phi,\nabla f_{ij}\rangle\Big)\gamma_{\sdimn}(x)dx.$$
So, as $r\to\infty$, $\abs{Q(\phi F,G)-Q(F,\phi G)}$ converges to $0$ by the Dominated Convergence Theorem and the Cauchy-Schwarz inequality, using $F,G\in\mathcal{F}$ and Definition \ref{def1}.
By the assumption \eqref{four75} on $F$, $Q(F,\phi G)\stackrel{\eqref{quaddef}}{=}\sum_{1\leq i<j\leq m}-\int_{\Sigma_{ij}}\phi g_{ij}L_{ij} f_{ij}\gamma_{\sdimn}(x)dx$.  So, as $r\to\infty$, $Q(F,\phi G)$ converges to $\langle-\L F,G\rangle$.  Therefore, as $r\to\infty$, $Q(\phi F,G)$ also converges to $\langle-\L F,G\rangle$.  The second assertion follows from the first, since $\abs{Q(\phi F,\phi G)-Q(F,\phi^{2}G)}$ converges to zero as $r\to\infty$ as well.
\end{proof}

\section{Curvature Bounds}

Below we denote $\Sigma\colonequals\cup_{1\leq i<j\leq m}\Sigma_{ij}$.

\begin{remark}\label{rk30}
Let $v\in\R^{\adimn}$.  For all $1\leq i<j\leq m$ let $f_{ij}\colon \Sigma_{ij}\to\R$ be defined by $f_{ij}\colonequals\langle v,N_{ij}\rangle$.  Then for all $1\leq i<j\leq m$, $L_{ij}f_{ij}=f_{ij}$ by Lemma \ref{lemma45} and Lemma \ref{lemma24}.  Also, the term in Lemma \ref{lemma29.9} is zero, since $X\colonequals v$ is the constant vector field in this case, i.e. \eqref{four75} holds.  For more detail on the latter, see \cite[(Eq. 3.13) and p. 476]{hutchings02}, where it is noted that $X\colonequals v$ corresponds to an infinitesimal translation with associated Jacobi functions $\{\langle N_{ij},v\rangle\}_{1\leq i<j\leq m}$.
\end{remark}

\begin{lemma}\label{lemma55}
Let $\Lambda$ be the set of solutions $\{\lambda_{ij}\}_{1\leq i<j\leq m}$ of the middle equation of \eqref{nine5.3}.  Then $\Lambda$ is a vector space of dimension equal to $m-1$.  Also, $\Lambda$ has an orthonormal basis (with respect to $\langle\cdot,\cdot\rangle$ defined in Lemma \ref{lemma25.3}) consisting of vectors all of whose components are nonzero.
\end{lemma}
\begin{proof}
From Lemma \ref{lemma25.3}, $\Lambda$ (i.e. $D_{2}$) has dimension equal to $m-1$ (since $D_{1}$ has dimension $\binom{m}{2}-m+1$ and $D_{2}\oplus D_{1}$ has dimension $\binom{m}{2}$).  Consider the sets described in Conjecture \ref{conj0}.  These sets satisfy $H_{ij}(x)=0$ for every $x\in\Sigma_{ij}$, and they also satisfy all equations from \eqref{nine5.3}, for any $y\in\R^{m-1}$.  We can then treat $N_{ij}$ as being constant functions of $y$, so that $\lambda_{ij}(y)=-\langle y,N_{ij}\rangle$ for all $1\leq i<j\leq m$ is a solution of the equations \eqref{nine5.3}.  By considering any $y\in\R^{m-1}$, linear algebra also implies then that $\Lambda$ has dimension at least $m-1$, since the only $y\in\R^{m-1}$ such that $\langle y,N_{ij}\rangle=0$ for all $1\leq i<j\leq m$ is $y=0$.  Finally, choosing an orthonormal basis of $y$'s of $\R^{m-1}$ so that each basis element is not perpendicular to $N_{ij}$ for all $1\leq i<j\leq m$, then we have $m-1$ nonvanishing solutions of  \eqref{nine5.3}.
\end{proof}

Since $Q$ defined in \eqref{quaddef} is a symmetric quadratic form by Lemma \ref{lemma32.5}, we anticipate that a function minimizing this quadratic form is an eigenfunction of $\L$.  When the minimum exists, we can get an eigenfunction of $\L$ in this way.  However, it is possible that the function minimizing $Q$ might change sign on connected components of $\Sigma$, contrary to our intuition that a fundamental tone should not change sign.  This sign changing property also causes problems for ensuing arguments, since the non-sign changing of the fundamental tone was crucial in the curvature bounds of \cite{colding12a,zhu16}.  To get around this issue, we instead minimize $Q$ over functions whose sign does not change on connected components of $\Sigma$.  By Lemma \ref{lemma55}, such a restriction is nontrivial.  It is still possible that such an $F$ minimizing $Q$ might vanish on the boundary of $\Sigma$.  We will deal with this issue in Lemmas \ref{lemma29} and \ref{rkorth}.

For any hypersurface $\Sigma\subset\R^{\adimn}$ (possibly with boundary), we define
\begin{equation}\label{seven0}
\pcon=\pcon(\Sigma)
\colonequals-\inf_{\substack{G\in \mathcal{F}\cap C_{0}^{\infty}(\Sigma)\colon\langle G,G\rangle=1,\\
G\,\,\mathrm{does}\,\,\mathrm{not}\,\,\mathrm{change}\,\,\mathrm{sign}\,\,\mathrm{on}\,\,\mathrm{any}\\
\,\,\mathrm{connected}\,\,\mathrm{component}\,\,\mathrm{of}\,\,\Sigma
}}Q(G,G).
\end{equation}
By the definition of $\pcon$,
\begin{equation}\label{seven1}
\Sigma_{1}\subset\Sigma_{2}\qquad\Longrightarrow\qquad\pcon(\Sigma_{1})\leq\pcon(\Sigma_{2}).
\end{equation}

%Below we denote $\pcon=\pcon(\Sigma)$ using \eqref{seven0}.

\begin{lemma}[\embolden{Existence of Fundamental Tone}]\label{lemma33}
Assume $\pcon\colonequals\pcon(\Sigma)<\infty$.  Then there exists $F\in\mathcal{F}$ such that
\begin{equation}\label{eqstar}
Q(F,F)=\min_{\substack{G\in\mathcal{F}\colon\langle G,G\rangle=1,\\
G\,\,\mathrm{does}\,\,\mathrm{not}\,\,\mathrm{change}\,\,\mathrm{sign}\,\,\mathrm{on}\,\,\mathrm{any}\\
\,\,\mathrm{connected}\,\,\mathrm{component}\,\,\mathrm{of}\,\,\Sigma}}Q(G,G).
\end{equation}
If $F\in\mathcal{F}$ satisfies \eqref{eqstar}, then the following hold.  $F$ is an eigenfunction of $\L$ so that
$$\L F=\pcon F.$$
Moreover, $\forall\,1\leq i<j<k\leq m$, $\forall$ $x\in(\redS_{ij})\cap(\redS_{jk})\cap(\redS_{ki})$, the following holds at $x$.
$$\nabla_{\nu_{ij}}f_{ij}+q_{ij}f_{ij}=\nabla_{\nu_{jk}}f_{jk}+q_{jk}f_{jk}=\nabla_{\nu_{ij}}f_{ij}+q_{ij}f_{ij}.$$
\end{lemma}
\begin{proof}
First note that the set of functions $G$ specified in \eqref{eqstar} is nonempty by Lemma \ref{lemma55}.

Fix $x$ in the interior of $\Sigma$.  Let $\Sigma_{1}\subset\Sigma_{2}\subset\ldots$ be a sequence of compact $C^{\infty}$ hypersurfaces (with boundary) such that $\cup_{k=1}^{\infty}\Sigma_{k}=\Sigma$.  For each $k\geq1$, let $F_{k}$ be a Dirichlet eigenfunction of $\L$ on $\Sigma_{k}$ such that $\L F_{k}=\pcon(\Sigma_{k})F_{k}$, and such that $F_{k}$ does not change sign on any connected component of $\Sigma$.  By multiplying by a constant, we may assume $F_{k}(x)=1$ for all $k\geq1$.  Since $\pcon(\Sigma_{k})$ increases to $\pcon(\Sigma)<\infty$ as $k\to\infty$ by \eqref{seven0}, the Harnack inequality implies that there exists $c=c(\Sigma_{k},\pcon(\Sigma))$ such that $1\leq\sup_{y\in B}F_{k}(y)\leq c\inf_{y\in B}F_{k}(y)\leq c$ for some neighborhood $B$ of $x$.  Elliptic theory then gives uniform $C^{2,\sigma}$ bounds for the functions $F_{1},F_{2},\ldots$ on each compact subset of $\Sigma$.  So, by Arzel\`{a}-Ascoli there exists a uniformly convergent subsequence of $F_{1},F_{2},\ldots$ which converges to a solution $LF=\pcon(\Sigma)F$ on $\Sigma$ with $F(x)=1$ such that $F$ does not change sign on any connected component of $\Sigma$.  The Harnack inequality then implies $F$ is nonzero on any connected component of $\Sigma$.

Let $G\in\mathcal{F}$.  For any $t\in\R$, define
$$c(t)\colonequals\frac{Q(F+tG,F+tG)}{\langle F+tG,F+tG\rangle}-Q(F,F).$$
By definition of $F$, we have $c(t)\geq0$ $\forall$ $t\in\R$.  Therefore, $c'(0)=0$.  By Lemma \ref{lemma32.5}, $Q(F,G)=Q(G,F)$, so
$$c(t)=\frac{Q(F,F)+2tQ(F,G)+t^{2}Q(G,G)}{\langle F,F\rangle+2t\langle F,G\rangle+t^{2}\langle G,G\rangle}-Q(F,F),$$
$$0=c'(0)=\frac{\langle F,F\rangle Q(F,Q)-Q(F,F)\langle F,G\rangle}{\langle F,F\rangle^{2}}.$$
Therefore, for any $G\in\mathcal{F}$,
$$Q(F,G)=\frac{Q(F,F)}{\langle F,F\rangle}\langle F,G\rangle.$$
Fix $1\leq i<j\leq m$.  Choosing $G$ smooth and localized away from $C$ then implies that $\L F=-\frac{Q(F,F)}{\langle F,F\rangle}F\equalscolon\pcon F$ on $\Sigma_{ij}$ for every $1\leq i<j\leq m$, away from their boundaries.  Fix $1\leq i<j<k\leq m$.  Choose the vector field $X$ (where $g_{pq}\colonequals\langle X,N_{pq}\rangle$ for all $1\leq p<q\leq m$) now such that $g_{ij}=-g_{jk}=1$ and $g_{ki}=0$ at some $x\in (\redS_{ij})\cap (\redS_{jk})\cap (\redS_{ki})$, and such that $X$ is supported in a neighborhood of $x$.  Then the definition of $Q(F,G)$ implies that $\nabla_{\nu_{ij}}f_{ij}+q_{ij}f_{ij}=\nabla_{\nu_{jk}}f_{jk}+q_{jk}f_{jk}$.  (This argument is valid as long as the sign of $f_{ij},f_{jk},f_{ki}$ are not all the same at $x$.  It cannot occur that all three of these numbers have the same sign, since they must sum to zero at $x$, and by the definition of $F$, these three functions cannot all have the same sign in a neighborhood of $x$, by the limiting definition of $F$.)

\end{proof}

\begin{lemma}\label{lemma23}
Let $\Omega_{1},\ldots,\Omega_{m}$ minimize Problem \ref{prob1}.  Then $\pcon\geq1$.
\end{lemma}
\begin{proof}
By Lemma \ref{lemma55}, there exist a linearly independent set of functions $F_{1},\ldots,F_{m-1}\colon\Sigma\to\R$ such that for any $1\leq i<j\leq m$ and for any $1\leq p\leq m-1$, $F_{p}$ is a nonzero constant on $\Sigma_{ij}$, and also
$$\sum_{p=1}^{m-1}(F_{p}(x))^{2}=1,\qquad\forall\,x\in\Sigma.$$
Fix $1\leq p\leq m-1$.  Let $\phi\in C_{0}^{\infty}(\Sigma)$.  Recall from Definition \ref{def2} that $Q(F_{p},\phi F_{p})$ is the sum of two terms.  Consider the second such term.  If we sum that term over all $1\leq p\leq m-1$, we get zero, since for any $1\leq i<j<k\leq m$ and for any $x\in(\redS_{ij})\cap(\redS_{jk})\cap(\redS_{ki})$, we have $q_{ij}(x)+q_{jk}(x)+q_{ki}(x)=0$ by Remark \ref{rk5}.  Therefore, there must exist some $1\leq p\leq m-1$ such that the second term of $Q(F_{p},\phi F_{p})$ is nonpositive.  That is, there exists $1\leq p\leq m-1$ such that
$$Q(F_{p},\phi F_{p})\leq-\sum_{1\leq i<j\leq m}\int_{\Sigma_{ij}}\phi (\vnormt{A}^{2}+1)(F_{p}(x))^{2}\gamma_{\sdimn}(x)dx.$$
(Since $F_{p}$ is constant, the $\nabla F_{p}$ term is zero in $Q(F_{p},\phi F_{p})$.)  But then, letting $\phi$ increase monotonically to $1$,
$$
\frac{Q(F_{p},F_{p})}{\langle F_{p},F_{p}\rangle}
\leq\frac{-\sum_{1\leq i<j\leq m}\int_{\Sigma_{ij}}(\vnormt{A}^{2}+1)(F_{p}(x))^{2}\gamma_{\sdimn}(x)dx}
{\sum_{1\leq i<j\leq m}\int_{\Sigma_{ij}}(F_{p}(x))^{2}\gamma_{\sdimn}(x)dx}.
$$
Since $F_{p}$ is nonvanishing, we conclude that $\pcon\geq1$.
\end{proof}

A key step of the important Lemma \ref{rkorth} below is to bound the gradient of $F$ by $F$ itself, where $F$ is an eigenfunction of $\L$.  As in \cite{colding12a,zhu16}, such a bound can result from a bound on the gradient of the logarithm of $F$, together with the Cauchy-Schwarz inequality.  Unfortunately, the function $F$ from Lemma \ref{lemma33} might vanish on the boundary of $\Sigma$.  So, we have to careful when bounding the gradient of the logarithm of $F$.  The argument below is adapted from \cite[Lemma 9.15(2)]{colding12a}.

\begin{lemma}\label{lemma29}
Assume $\pcon\colonequals\pcon(\Sigma)<\infty$.  Suppose $F\in\mathcal{F}$, $LF=\pcon F$ and $F$ satisfies \eqref{four75}.  Let $\phi\in C_{0}^{\infty}(\Sigma)$ such that
$$\{x\in\Sigma\colon \phi(x)=0\}\supset\{x\in\Sigma\colon\exists\,\,1\leq i<j\leq m\,\,\mathrm{such}\,\,\mathrm{that}\,\, f_{ij}(x)=0\}.$$
Then
\begin{flalign*}
&\sum_{1\leq i<j\leq m}\int_{\Sigma_{ij}} \phi^{2}\Big(\vnormt{A}^{2}+\vnormt{\nabla \log \abs{f_{ij}}}^{2}\Big)\gamma_{\sdimn}(x)dx
+\int_{\redS_{ij}}\phi^{2}\nabla_{\nu_{ij}}\log \abs{f_{ij}}\gamma_{\sdimn}(x)dx\\
&\qquad\leq4\sum_{1\leq i<j\leq m}\int_{\Sigma_{ij}}\vnormt{\nabla\phi}^{2}+(\pcon-1)\phi^{2}\,\gamma_{\sdimn}(x)dx.
\end{flalign*}
\end{lemma}
\begin{proof}
Fix $1\leq i<j\leq m$.  On the interior of $\Sigma_{ij}$, if $f_{ij}\neq0$, we have
\begin{flalign*}
\mathcal{L}_{ij}\log \abs{f_{ij}}
&\stackrel{\eqref{three4.3}\wedge\eqref{seven00p}}{=}\sum_{k=1}^{\sdimn}\nabla_{e_{k}}\left(\frac{\nabla_{e_{k}}f_{ij}}{f_{ij}}\right)-\frac{\langle x,\nabla f_{ij}\rangle}{f_{ij}}
=\sum_{k=1}^{\sdimn}\frac{-(\nabla_{e_{k}}f_{ij})^{2}}{f_{ij}^{2}}+\frac{\mathcal{L}_{ij}f_{ij}}{f_{ij}}\\
&=-\vnormt{\nabla\log \abs{f_{ij}}}^{2}+\frac{\mathcal{L}_{ij}f_{ij}}{f_{ij}}
\stackrel{\eqref{three4.5}}{=}-\vnormt{\nabla\log \abs{f_{ij}}}^{2}+\frac{\L_{ij}f_{ij} -\vnormt{A}^{2}f_{ij}-f_{ij}}{f_{ij}}\\
&=-\vnormt{\nabla\log \abs{f_{ij}}}^{2}+\frac{\pcon f_{ij} -\vnormt{A}^{2}f_{ij}-f_{ij}}{f_{ij}}\\
&=-\vnormt{\nabla\log \abs{f_{ij}}}^{2}+(\pcon-1)-\vnormt{A}^{2}.
\end{flalign*}

Let $\nu_{ij}$ denote the unit exterior normal vector to the boundary of $\Sigma_{ij}$.  By Lemma \ref{lemma32.5},
% div(f grad g gam)= fLg gam+<grad f, grad g>gam
\begin{flalign*}
&\int_{\Sigma_{ij}}\langle\nabla \phi^{2},\nabla \log \abs{f_{ij}}\rangle\gamma_{\sdimn}(x)dx\\
&\stackrel{\eqref{three4.3}\wedge\eqref{seven00p}}{=}-\int_{\Sigma_{ij}}\phi^{2}\mathcal{L}\log \abs{f_{ij}}\gamma_{\sdimn}(x)dx+\int_{\redS_{ij}}\phi^{2}\nabla_{\nu_{ij}}\log \abs{f_{ij}}\gamma_{\sdimn}(x)dx\\
&\quad=\int_{\Sigma_{ij}}\phi^{2}\Big(\vnormt{\nabla\log \abs{f_{ij}}}^{2}+(1-\pcon)+\vnormt{A}^{2}\Big)\gamma_{\sdimn}(x)dx
+\int_{\redS_{ij}}\phi^{2}\nabla_{\nu}\log \abs{f_{ij}}\gamma_{\sdimn}(x)dx.
\end{flalign*}
(Since $\phi$ is zero whenever $f_{ij}=0$ by assumption, all expressions above are well-defined.)  By the arithmetic mean geometric mean inequality, %2ab\leq (a^2+b^2)
$$\abs{\langle\nabla \phi^{2},\nabla \log \abs{f_{ij}}\rangle}
\leq2\vnormt{\nabla\phi}^{2}+\frac{1}{2}\phi^{2}\vnormt{\nabla \log \abs{f_{ij}}}^{2}.$$
So, combining the above and summing over $1\leq i<j\leq m$,
\begin{flalign*}
&\sum_{1\leq i<j\leq m}\int_{\Sigma_{ij}} \phi^{2}\Big(\vnormt{A}^{2}+\vnormt{\nabla \log \abs{f_{ij}}}^{2}\Big)\gamma_{\sdimn}(x)dx
+\int_{\redS_{ij}}\phi^{2}\nabla_{\nu_{ij}}\log \abs{f_{ij}}\gamma_{\sdimn}(x)dx\\
&\qquad\leq4\sum_{1\leq i<j\leq m}\int_{\Sigma_{ij}}\vnormt{\nabla\phi}^{2}+(\pcon-1)\phi^{2}\gamma_{\sdimn}(x)dx.
\end{flalign*}
\end{proof}

The following curvature bound is adapted from \cite[Lemma 5.1]{mcgonagle15}.

\begin{lemma}\label{lemma29b}
Let $\Omega_{1},\ldots,\Omega_{m}$ minimize Problem \ref{prob1}.  Then $\forall$ $\phi\in C_{0}^{\infty}(\Sigma)$,
$$
\sum_{1\leq i<j\leq m}\int_{\Sigma_{ij}}\phi^{2}\vnorm{A}^{2}\gamma_{\sdimn}(x)dx
\leq \sum_{1\leq i<j\leq m}\int_{\Sigma_{ij}}\Big((\pcon-1)\phi^{2}+\vnorm{\nabla\phi}^{2}\Big)\gamma_{\sdimn}(x)dx.
$$
\end{lemma}
\begin{proof}
Let $G\colonequals\{\alpha_{ij}\}_{1\leq i<j\leq m}$ be a solution to the system of middle equations of \eqref{nine5.3}.  Let $\phi\in C_{0}^{\infty}(\R^{\adimn})$.  By the definition \eqref{seven0} of $\pcon$,
$$-Q(\phi G,\phi G)\leq\pcon\langle \phi G,\phi G\rangle.$$
That is, by Lemma \ref{lemma32.5}
\begin{flalign*}
&\sum_{1\leq i<j\leq m}\int_{\Sigma_{ij}}\alpha_{ij}^{2}(-\vnorm{\nabla\phi}^{2}+\phi^{2}(\vnorm{A}^{2}+1))\gamma_{\sdimn}(x)dx\\
&\qquad\qquad
+\sum_{1\leq i<j<k\leq m}\int_{\redS_{ij}\cap\redS_{jk}\cap \redS_{ki}}
\phi^{2}\Big(q_{ij}\alpha_{ij}^{2}+q_{jk}\alpha_{jk}^{2}+q_{ki}\alpha_{ki}^{2}\Big)\gamma_{\sdimn}(x)\\
&\qquad\leq \pcon\sum_{1\leq i<j\leq m}\int_{\Sigma_{ij}}\phi^{2}\alpha_{ij}^{2}\gamma_{\sdimn}(x)dx.
\end{flalign*}
Summing these quantities over all permutations of $\{1,\ldots,m\}$, i.e. permuting $\{\alpha_{ij}\}_{1\leq i<j\leq m}$, the middle term vanishes by Remark \ref{rk5}, and we get
$$
\sum_{1\leq i<j\leq m}\int_{\Sigma_{ij}}(-\vnorm{\nabla\phi}^{2}+\phi^{2}(\vnorm{A}^{2}+1))\gamma_{\sdimn}(x)dx
\leq \pcon\sum_{1\leq i<j\leq m}\int_{\Sigma_{ij}}\phi^{2}\gamma_{\sdimn}(x)dx.
$$
Rearranging completes the proof.

\end{proof}

%\begin{itemize}
%%\item We say that $f\in L_{2}(\Sigma,\gamma_{\sdimn})$ if $\int_{\Sigma}\abs{f}^{2}\gamma_{\sdimn}(x)dx<\infty$.
%\item We say that $f\in W_{1,2}(\Sigma,\gamma_{\sdimn})$ if $\int_{\Sigma}(\abs{f}^{2}+\vnormt{\nabla f}^{2})\gamma_{\sdimn}(x)dx<\infty$.
%%\item We say that $f\in W_{2,2}(\Sigma,\gamma_{\sdimn})$ if $\int_{\Sigma}(\abs{f}^{2}+\vnormt{\nabla f}^{2}+\abs{\mathcal{L}f}^{2})\gamma_{\sdimn}(x)dx<\infty$.
%\end{itemize}

\begin{lemma}[{\cite[Lemma 6.2]{zhu16}}]\label{lemma28.5}
Let $\Omega_{1},\ldots,\Omega_{m}$ minimize Problem \ref{prob1}.  If $\int_{\Sigma}(\abs{\phi}^{2}+\vnormt{\nabla \phi}^{2})\gamma_{\sdimn}(x)dx<\infty$ and if $\phi$ is bounded, then

$$\int_{\Sigma}\phi^{2}\vnormt{A}^{2}\gamma_{\sdimn}(x)dx
\leq \int_{\Sigma}(\vnormt{\nabla\phi}^{2}+(\pcon-1)\phi^{2})\gamma_{\sdimn}(x)dx.$$
\end{lemma}
\begin{proof}
%By Lemma \ref{lemma52.6}, Assumption \ref{as1} holds, so the Hausdorff dimension of the boundary of $\Sigma$ is at most $\sdimn-2$.

Apply Lemma \ref{lemma29b}, Lemma \ref{lemma60} and Fatou's Lemma.
\end{proof}

A rearrangement argument implies the following decay for the Gaussian surface area of optimal sets far from the origin.

\begin{lemma}[{\cite[Lemma 4.3]{milman18a}}]\label{lemma28.8}
Let $\Omega_{1},\ldots,\Omega_{m}$ minimize Problem \ref{prob1}.  Then there exists $r_{m}>0$ so that, for all $r>r_{m}$,
$$\sum_{1\leq i<j\leq m}\gamma_{\sdimn}(\Sigma_{ij}\cap\{x\in\R^{\adimn}\colon\vnorm{x}>r\})
\leq 3m\gamma_{\sdimn}(\{x\in\R^{\adimn}\colon\vnorm{x}=r\}).$$
\end{lemma}

The following Lemmas follow from Lemma \ref{lemma28.5}.

\begin{lemma}\label{lemma61}
Let $\Omega_{1},\ldots,\Omega_{m}$ minimize Problem \ref{prob1}.  Then
$$\sum_{1\leq i<j\leq m}\int_{\Sigma_{ij}}\vnormt{A}^{2}\gamma_{\sdimn}(x)dx<\infty.$$
Consequently, for any $v\in\R^{\adimn}$, by \eqref{three2},
$$\sum_{1\leq i<j\leq m}\int_{\Sigma_{ij}}\vnorm{\nabla\langle v,N\rangle}^{2}\gamma_{\sdimn}(x)dx<\infty$$
\end{lemma}
\begin{proof}
Use Lemma \ref{lemma28.5} and \eqref{three2}.
\end{proof}

\section{Orthogonality of Eigenfunctions}

The following is an adaptation of \cite[Lemma 9.44]{colding12a} and \cite[Proposition 6.11]{zhu16}.

\begin{lemma}[\embolden{Orthogonality of Eigenfunctions}]\label{rkorth}
Suppose Assumption \ref{as2} holds.  Let $\pcon\colonequals\pcon(\{x\in\Sigma\colon\vnorm{x}\leq t\})$.  Assume $1<\pcon<\infty$.  Let $F\in\mathcal{F}$ be the (Dirichlet) eigenfunction of $\{x\in\Sigma\colon\vnorm{x}\leq t\}$ that exists by Lemma \ref{lemma33}, so that $\L F=\pcon F$ and $F=0$ when $\vnorm{x}=t$ (so that $F$ has Dirichlet boundary conditions).  Fix $v\in\R^{\adimn}$.  Let $G\in\mathcal{F}$ so that $g_{ij}=\langle v,N_{ij}\rangle$ for all $1\leq i<j\leq m$.  Then for all $\epsilon>0$, there exists $t>0$ such that
$$\abs{\langle F,G\rangle}<\epsilon(\pcon-1)^{-1/2}\langle F,F\rangle^{1/2}\langle G,G\rangle^{1/2}.$$
Here $\epsilon$ does not depend on $\pcon$.
\end{lemma}
\begin{remark}
The following informal argument would show that $\langle F,G\rangle=0$, though since $\Sigma$ is not compact and has boundary, this informal argument is not rigorous.  By the definition of $\mathcal{F}$ and our assumption, the second term in the definition of $Q(F,G)$ is zero, so that $Q(F,G)=-\langle\L F,G\rangle$, so by Lemma \ref{lemma32.5}
$$\pcon\langle F,G\rangle=\langle \L F,G\rangle=\langle F,\L G\rangle=\langle F,G\rangle.$$
Since $\pcon\neq1$, we have $\langle F,G\rangle=0$.
\end{remark}
\begin{proof}
By Lemma \ref{lemma33} and Lemma \ref{lemma29.9} (using that the vector field $X\colonequals v$ is constant for $G$), $\forall\,1\leq i<j<k\leq m$, and for all $x\in (\redS_{ij})\cap (\redS_{jk})\cap (\redS_{ki})$, we have
\begin{equation}\label{five7.3}
\begin{aligned}
\nabla_{\nu_{ij}}f_{ij}+q_{ij}f_{ij}&=\nabla_{\nu_{jk}}f_{jk}+q_{jk}f_{jk}=\nabla_{\nu_{ij}}f_{ij}+q_{ij}f_{ij}.\\
\nabla_{\nu_{ij}}g_{ij}+q_{ij}g_{ij}&=\nabla_{\nu_{jk}}g_{jk}+q_{jk}g_{jk}=\nabla_{\nu_{ij}}g_{ij}+q_{ij}g_{ij}.\\
\end{aligned}
\end{equation}
By \eqref{seven00}, \eqref{seven00p}, \eqref{one3.5n} and \eqref{one3.5o},
\begin{equation}\label{five6}
 G\mL F-F\mL G =G\L F-F\L G=(\pcon-1)FG.
 \end{equation}
 % div \phi X = phi divX + \nabla \phi X

Let $\psi\in C_{0}^{\infty}(\Sigma)$.  Then
$$
\mathrm{div}_{\tau}[(G\nabla F-F\nabla G)\gamma_{\sdimn}]
=(G\mL F- F\mL G)\gamma_{\sdimn}
\stackrel{\eqref{five6}}{=}(\pcon-1)FG\gamma_{\sdimn}.
$$
Integrating by parts gives
\begin{equation}\label{five7.6}
\begin{aligned}
(\pcon-1)\langle\psi F,G\rangle
&\stackrel{\eqref{ipdef}}{=}\sum_{1\leq i<j\leq m}\int_{\Sigma_{ij}}(\pcon-1)\psi f_{ij}g_{ij}\gamma_{\sdimn}(x)dx\\
&=\sum_{1\leq i<j\leq m}\int_{\Sigma_{ij}}\langle\nabla\psi, g_{ij}\nabla f_{ij}-f_{ij}\nabla g_{ij}\rangle\gamma_{\sdimn}(x)dx\\
&\qquad\qquad+\int_{\redS_{ij}}\psi(g_{ij}\nabla_{\nu_{ij}}f_{ij}-f_{ij}\nabla_{\nu_{ij}}g_{ij})\gamma_{\sdimn}(x)dx.
\end{aligned}
\end{equation}

Adding and subtracting the same terms, and using Definition \ref{def1} and \eqref{five7.3},
\begin{flalign*}
&\sum_{1\leq i<j\leq m}\int_{\redS_{ij}}\psi(g_{ij}\nabla_{\nu_{ij}}f_{ij}-f_{ij}\nabla_{\nu_{ij}}g_{ij})\gamma_{\sdimn}(x)dx\\% -qfg+qfg
&\qquad=\sum_{1\leq i<j\leq m}\int_{\redS_{ij}}\psi(g_{ij}[\nabla_{\nu_{ij}}f_{ij}+q_{ij}f_{ij}]
-f_{ij}[\nabla_{\nu_{ij}}g_{ij}+q_{ij}g_{ij}])\gamma_{\sdimn}(x)dx
\stackrel{\eqref{five7.3}}{=}0.
\end{flalign*}

We can then rewrite \eqref{five7.6} as
\begin{equation}\label{five7.6p}
(\pcon-1)\langle\psi F,G\rangle
=\sum_{1\leq i<j\leq m}\int_{\Sigma_{ij}}\langle\nabla\psi, g_{ij}\nabla f_{ij}-f_{ij}\nabla g_{ij}\rangle\gamma_{\sdimn}(x)dx.
\end{equation}

We split this integral into two pieces.  For any $1\leq i<j\leq m$ let $B_{ij}\subset\Sigma_{ij}$ and denote $B\colonequals \cup_{1\leq i<j\leq m}B_{ij}$.   Define $C$ as in Definition \ref{defnote}.  For any $r,\epsilon>0$, denote
$$B(0,r)\colonequals\{x\in\R^{\adimn}\colon\vnormt{x}\leq r\},$$
$$C_{\epsilon,r}\colonequals\{x\in\R^{\adimn}\colon\exists\,c\in C\,\,\mathrm{such}\,\,\mathrm{that}\,\, \vnorm{x-c}<\epsilon\,\,\mathrm{and}\,\,r<\vnorm{c}<r+1\}.$$
So, $C_{\epsilon,r}$ is the $\epsilon$-neighborhood of the part of $C$ lying in the annulus with radii $r$ and $r+1$.  By Assumption \ref{as2}, let $\eta\in C_{0}^{\infty}(\R^{\adimn})$ such that $\eta=0$ on $C_{\epsilon,r}$, $0\leq\eta\leq1$ everywhere, $\eta=1$ on $C_{2\epsilon,r}^{c}$.  We then write \eqref{five7.6p} as
\begin{equation}\label{five9}
\begin{aligned}
(\pcon-1)\langle\psi F,G\rangle
&=\sum_{1\leq i<j\leq m}\int_{\Sigma_{ij}}\eta\langle\nabla\psi, g_{ij}\nabla f_{ij}-f_{ij}\nabla g_{ij}\rangle\gamma_{\sdimn}(x)dx\\
&\qquad\sum_{1\leq i<j\leq m}\int_{\Sigma_{ij}}(1-\eta)\langle\nabla\psi, g_{ij}\nabla f_{ij}-f_{ij}\nabla g_{ij}\rangle\gamma_{\sdimn}(x)dx.
\end{aligned}
\end{equation}

We estimate the first term in \eqref{five9}.   Define $\phi\colonequals\eta\vnorm{\nabla\psi}$.  Using the Cauchy-Schwarz inequality on the first term of \eqref{five9},
\begin{flalign*}
&\abs{\sum_{1\leq i<j\leq m}\int_{\Sigma_{ij}}\eta\langle\nabla\psi, g_{ij}\nabla f_{ij}-f_{ij}\nabla g_{ij}\rangle\gamma_{\sdimn}(x)dx}\\
&\qquad\leq\sum_{1\leq i<j\leq m}\int_{\Sigma_{ij}}\phi \Big(\abs{g_{ij}}\vnorm{\nabla f_{ij}}+\abs{f_{ij}}\vnorm{\nabla g_{ij}}\Big)\gamma_{\sdimn}(x)dx\\
&\qquad\stackrel{\eqref{three2}}{\leq}\sum_{1\leq i<j\leq m}\int_{\Sigma_{ij}}\phi \Big(\vnorm{v}\abs{f_{ij}}\frac{\vnorm{\nabla f_{ij}}}{\abs{f_{ij}}}+\abs{f_{ij}}\vnorm{v}\vnormt{A}\Big)\gamma_{\sdimn}(x)dx\\
&\qquad\leq\vnorm{v}\langle F,F\rangle^{1/2}
\Big(\sum_{1\leq i<j\leq m}\int_{\Sigma_{ij}}\phi^{2}\Big(\vnorm{\nabla\log\abs{f_{ij}}}^{2}+\vnormt{A}^{2}\Big)\gamma_{\sdimn}(x)dx\Big)^{1/2}.
\end{flalign*}

Note that $\phi=0$ on $C$, so the boundary term in Lemma \ref{lemma29} is zero.  We therefore apply Lemma \ref{lemma29} to get
\begin{flalign*}
&\abs{\sum_{1\leq i<j\leq m}\int_{\Sigma_{ij}}\eta\langle\nabla\psi, g_{ij}\nabla f_{ij}-f_{ij}\nabla g_{ij}\rangle\gamma_{\sdimn}(x)dx}\\
&\qquad\leq2\vnorm{v}\langle F,F\rangle^{1/2}\Big(\sum_{1\leq i<j\leq m}\int_{\Sigma_{ij}}\vnorm{\nabla\phi}^{2}+(\pcon-1)\phi^{2}\gamma_{\sdimn}(x)dx\Big)^{1/2}.
\end{flalign*}

Note that $1-\eta$ is only nonzero on $C_{2\epsilon,r}$, so we can estimate \eqref{five9} as
\begin{flalign*}
\abs{(\pcon-1)\langle\psi F,G\rangle}
&\leq 2\vnorm{v}\langle F,F\rangle^{1/2}\Big(\sum_{1\leq i<j\leq m}\int_{\Sigma_{ij}}\vnorm{\nabla\phi}^{2}+(\pcon-1)\phi^{2}\gamma_{\sdimn}(x)dx\Big)^{1/2}\\
&\qquad+\sum_{1\leq i<j\leq m}\int_{C_{2\epsilon,r}\cap \Sigma_{ij}}\big(\vnorm{\nabla\psi} \vnorm{v}\vnorm{\nabla f_{ij}}+\abs{f_{ij}}\vnorm{v}\vnorm{A}\big)\gamma_{\sdimn}(x)dx.
\end{flalign*}
The latter term is negligible when $\epsilon>0$ is small, since the integrand does not depend on $\epsilon$, whereas the region of integration becomes smaller as $\epsilon\to0^{+}$.  The other integral is negligible by Assumption \ref{as2}, and then letting $r=t\to\infty$ completes the proof, since
%$$
%\abs{\langle\psi F,G\rangle}
%\leq 10(\pcon-1)^{-1/2}\vnorm{v}\langle F,F\rangle^{1/2}\Big(\gamma_{\sdimn-1}(C_{0,r})+\gamma_{\sdimn}(\Sigma\cap B(0,r+1)\setminus B(0,r-1))\Big)^{1/2}.
%$$
%The last quantity decays exponentially in $r$ as $r\to\infty$ by Lemma \ref{lemma28.8}.  The quantity $\gamma_{\sdimn-1}(C_{0,r})$ also decays exponentially in $r$ as $r\to\infty$ by Assumption \ref{as2}.
$1-\psi$ is only nonzero on $B(0,r)^{c}$, so
$$\abs{\langle F,G\rangle }
\leq \abs{\langle \psi F,G\rangle }+\abs{\langle F,G(1-\psi)\rangle }
\leq \abs{\langle \psi F,G\rangle }+\langle F,F\rangle^{1/2}\langle G,G1_{B(0,r)^{c}}\rangle^{1/2}.$$
And the last quantity goes to $0$ as $r\to\infty$ by the Dominated Convergence Theorem.

\end{proof}

\section{Dimension Reduction}\label{secred}

Below, we fill in the details to the argument sketched in the introductory Section \ref{intsec}.

\begin{proof}[Proof of Theorem \ref{thm0}]
For any $v\in\R^{\adimn}$, define
$$T(v)\colonequals\Big(\int_{\redb\Omega_{1}}\sum_{j\in\{1,\ldots,m\}\colon j\neq 1}\langle v,N_{1j}\rangle\gamma_{\sdimn}(x)dx,
\ldots,\int_{\redb\Omega_{m}}\sum_{j\in\{1,\ldots,m\}\colon j\neq m}\langle v,N_{mj}\rangle\gamma_{\sdimn}(x)dx\Big).$$
Then $T\colon\R^{\adimn}\to\R^{m}$ is linear.  By the rank-nullity theorem, the dimension of the kernel of $T$ plus the dimension of the image of $T$ is $\adimn$.  Since the sum of the indices of $T(v)$ is zero for any $v\in\R^{\adimn}$ (since $N_{ij}=-N_{ji}$ $\forall$ $1\leq i<j\leq m$ by Definition \ref{defnote}), the dimension $\ell$ of the image of $T$ is at most $m-1$.

Let $v$ in the kernel of $T$.  For any $1\leq i<j\leq m$, let $f_{ij}\colonequals\phi\langle v,N_{ij}\rangle$.  Let $X\colonequals \phi v$ be the chosen vector field.  Since $\Omega_{1},\ldots,\Omega_{m}$ minimize Problem \ref{prob1},
$$0\leq \frac{d^{2}}{ds^{2}}|_{s=0}\sum_{1\leq i<j\leq m}\int_{\Sigma_{ij}^{(s)}}\gamma_{\sdimn}(x)dx.$$
From Lemmas \ref{lemma28}, \ref{lemma29.9}, \ref{lemma97}, \ref{lemma27}, and then letting $\phi$ increase monotonically to $1$ (as in Lemma \ref{lemma97}),
$$0\leq\sum_{1\leq i<j\leq m}-\int_{\Sigma_{ij}}f_{ij}L_{ij}f_{ij}\gamma_{\sdimn}(x)dx.$$
(By Lemma \ref{lemma97} with $F=\{\langle v, N_{ij}\rangle\}_{1\leq i<j\leq m}$, the second term in \eqref{quaddef} for $Q(\phi F,\phi F)$ is zero in the limit as $\phi$ increases to $1$.)  By Remark \ref{rk30},
$$0\leq\sum_{1\leq i<j\leq m}-\int_{\Sigma_{ij}}f_{ij}^{2}\gamma_{\sdimn}(x)dx.$$
The last quantity must then be zero.  In summary, for any $v$ in the kernel of $T$, $\forall$ $1\leq i<j\leq m$, $f_{ij}(x)=\langle v,N_{ij}(x)\rangle=0$ for all $x\in\Sigma_{ij}$.  That is, $\exists$ $0\leq\ell\leq m-1$ as stated in the conclusion of Theorem \ref{thm0}.  It remains to characterize $\ell$ in terms of the Gaussian centers of mass of the sets $\Omega_{1},\ldots,\Omega_{m}$.

The image of $T$ is the span of
$$\Big\{\Big(\int_{\redb\Omega_{1}}\sum_{\substack{j\in\{1,\ldots,m\}\colon\\ j\neq 1}}\langle v,N_{1j}\rangle\gamma_{\sdimn}(x)dx,
\ldots,\int_{\redb\Omega_{m}}\sum_{\substack{j\in\{1,\ldots,m\}\colon\\ j\neq m}}\langle v,N_{mj}\rangle\gamma_{\sdimn}(x)dx\Big)\in\R^{m}\colon v\in\R^{\adimn}\Big\}.$$
Using the Divergence Theorem for $\R^{\adimn}$, if $v\in\R^{\adimn}$, then
$$
\int_{\redb\Omega_{1}}\sum_{j\in\{1,\ldots,m\}\colon j\neq 1}\langle v,N_{1j}\rangle\gamma_{\sdimn}(x)dx
=-\int_{\Omega_{1}}\langle v,x\rangle\gamma_{\sdimn}(x)dx,
$$
and similarly for $\Omega_{2},\ldots,\Omega_{m}$.  So, the image of $T$ is the span of
$$\Big\{\Big(\int_{\Omega_{1}}x\gamma_{\sdimn}(x)dx,\ldots,\int_{\Omega_{m}}x\gamma_{\sdimn}(x)dx\Big)^{t}v\in\R^{m}\colon v\in\R^{\adimn}\Big\}.$$
Here $(\int_{\Omega_{1}}x\gamma_{\sdimn}(x)dx,\ldots,\int_{\Omega_{m}}x\gamma_{\sdimn}(x)dx)^{t}$ denotes the matrix with $m$ rows, each of which is in $\R^{\adimn}$.
So, the dimension $\ell$ of the image of $T$ is equal to the dimension of the span of the vectors $\int_{\Omega_{1}}x\gamma_{\adimn}(x)dx,\ldots,\int_{\Omega_{m}}x\gamma_{\adimn}(x)dx$, as desired.
\end{proof}

\section{Proof of Structure Theorem Dichotomy}\label{secdi}

As above, $\Sigma\colonequals\cup_{1\leq i<j\leq m}\Sigma_{ij}$ and $\Sigma_{ij}\colonequals(\redb\Omega_{i})\cap(\redb\Omega_{j})$, $\forall$ $1\leq i<j\leq m$, and we define
$$
\pcon=\pcon(\Sigma)
\colonequals-\inf_{\substack{G\in \mathcal{F}\cap C_{0}^{\infty}(\Sigma)\colon\langle G,G\rangle=1,\\
G\,\,\mathrm{does}\,\,\mathrm{not}\,\,\mathrm{change}\,\,\mathrm{sign}\,\,\mathrm{on}\,\,\mathrm{any}\\
\,\,\mathrm{connected}\,\,\mathrm{component}\,\,\mathrm{of}\,\,\Sigma
}}Q(G,G).
$$

\begin{proof}[Proof of Theorem \ref{thm1}]
%Let $m\geq4$.
Existence of $\Omega_{1},\ldots,\Omega_{m}$ follows from Lemma \ref{lemma51p}.  Assumption \ref{as1} holds by Lemma \ref{lemma52.6}.  Assume that the second case of Theorem \ref{thm1} does not occur.  That is, assume that $\ell=m-1$ is the smallest possible $\ell$ such that the second case of Theorem \ref{thm1} holds (recalling that $\ell\leq m-1$ by Theorem \ref{thm0}).  If $\pcon=1$, it follows from Lemma \ref{lemma29b} that $\vnorm{A}=0$ on $\Sigma$.  So, either $\pcon>1$ or $\vnorm{A}=0$ on $\Sigma$ (since $\pcon\geq1$ by Lemma \ref{lemma23}).

Suppose $2\geq\pcon>1$.  Let $t>0$, let $\Sigma_{t}\colonequals\{x\in\Sigma\colon\vnorm{x}\leq t\}$, let $\pcon_{t}=\pcon(\Sigma_{t})$ and let $F=F_{t}$ from Lemma \ref{lemma33} so that $\L F=\pcon_{t} F$ and so that $F$ satisfies \eqref{four75}.  For $t$ sufficiently large, $\pcon_{t}>1$ by definition of $\pcon$.  If $F$ is volume-preserving, i.e. if \eqref{eight2} holds, then by Lemma \ref{lemma28}, we get a contradiction by Lemmas \ref{lemma29.9} and \ref{lemma97}.  Otherwise, fix $v\in\R^{\adimn}$ and let $G$ so that $g_{ij}\colonequals\langle v,N_{ij}\rangle$ for all $1\leq i<j\leq m$.  Since $\ell=m-1$, we can choose $v\neq0$ such that $F+G$ satisfies \eqref{eight2}.  By Remark \ref{rk30} and Lemma \ref{lemma29.9}, \eqref{four75} holds for $G$.  Then, by Lemma \ref{lemma97} we have
$$Q(F+G,\phi(F+G))=\langle -L(F+G),\phi(F+G)\rangle=-\langle\pcon_{t} F+G,\phi(F+G)\rangle.$$
Letting $\phi$ increase monotonically to $1$ (as in Lemma \ref{lemma97}), we then get
$$Q(F+G,F+G)
=-\langle\pcon_{t} F+G,F+G\rangle
=-\pcon_{t}\langle F,F\rangle-\langle G,G\rangle-(\pcon_{t}+1)\langle F,G\rangle<0.$$
To see that the last expression is negative, choose $\epsilon>0$ and $t$ large by Lemma \ref{rkorth} so that $\epsilon<1/3\leq 1/(\pcon_{t}+1)$, so that $\abs{(\pcon_{t}+1)\langle F,G\rangle}<\langle F,F\rangle^{1/2}\langle G,G\rangle^{1/2}$.  Then
$$Q(F+G,F+G)\leq (1-\pcon_{t})\langle F,F\rangle-(\langle F,F\rangle^{1/2}-\langle G,G\rangle^{1/2})^{2}<0.$$
So, we have contradicted Lemma \ref{lemma28} and minimality of $\Omega_{1},\ldots,\Omega_{m}$.

In the remaining case that $\pcon>2$ Lemma \ref{rkorth} is not needed.  Let $t>0$ so that $\pcon_{t}>2$.  Let $F$ be a Dirichlet eigenfunction such that $\L F=\pcon_{t} F$.  We then repeat the same computation, and use Lemma \ref{lemma32.5} to get
\begin{flalign*}
Q(F+G,F+G)
&= Q(F,F)+2Q(G,F)+Q(G,G)
=-\pcon_{t}\langle F,F\rangle-2\langle F,G\rangle-\langle G,G\rangle\\
&=-(\pcon_{t}-2)\langle F,F\rangle-\langle F+G,F+G\rangle<0.
\end{flalign*}
\end{proof}

\section{Proof of (Conditional) Double Bubble Problem}\label{secdoub}

\begin{figure}[h]
\centering
\def\svgwidth{.4\textwidth}
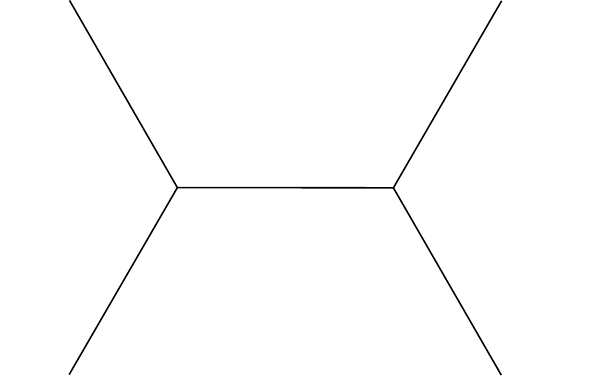
\caption{Proof of Corollary \ref{cor1}.}
\end{figure}

\begin{proof}[Proof of Corollary \ref{cor1}]
Let $m=3$.  Existence of $\Omega_{1},\Omega_{2},\Omega_{3}$ follows from Lemma \ref{lemma51p}.  Assumption \ref{as1} holds by Lemma \ref{lemma52.6}.  By Theorem \ref{thm0}, we may assume that $\Omega_{1},\Omega_{2},\Omega_{3}\subset\R^{2}$.  From Theorem \ref{thm1}, either all connected components of $\redb\Omega_{1},\redb\Omega_{2},\redb\Omega_{3}$ are flat lines, or we may assume that $\Omega_{1},\Omega_{2},\Omega_{3}\subset\R$.  In the second case, a rearrangement argument implies that $\Omega_{1},\Omega_{2},\Omega_{3}$ must be connected intervals, so we now consider the first case.

In the first case, we claim that $\Omega_{1},\Omega_{2},\Omega_{3}$ must have exactly three connected components.  We show this by contradiction.  Suppose $\Omega_{1},\Omega_{2},\Omega_{3}$ have more than three connected components.  By Assumption \ref{as1}, there must then exist two points $x,y\in\R^{2}$ such that (after relabeling the sets), $\Omega_{1}\cap\Omega_{2}$ contains the line segment $\ell_{0}$ between $x$ and $y$, and $\Omega_{3}$ intersects both $x$ and $y$.  Let $v\in\R^{2}$ be nonzero and parallel to $\ell_{0}$.  By Assumption \ref{as1} (which holds by Lemma \ref{lemma52.6}), the boundaries of $\Omega_{1},\Omega_{2},\Omega_{3}$ consist of line segments pointing in only three directions.  So, any line segment not parallel to $\ell_{0}$ is not parallel to $v$.  We can then put the edges of $\partial\Omega_{1}\cup\partial\Omega_{2}\cup\partial\Omega_{3}$ into at least two nonempty equivalent classes.  An edge $e$ is in the equivalence class labelled as $x$ if a sequence of adjacent edges can be constructed connecting $x$ to $e$ such that $v$ is not parallel to any edge in the sequence.  Similarly, an edge $e$ is in the equivalence class labelled as $y$ if a sequence of adjacent edges can be constructed connecting $y$ to $e$ such that $v$ is not parallel to any edge in the sequence.

We now claim we have linearly independent eigenfunctions of $\L$.  These correspond to (i) the vector field $X\colonequals v$, (ii) the constant vector field $X$ that is orthogonal to $v$, and (iii) to the function on $\Sigma$ such that $X=v$ everywhere, while $X=-v$ on edges in the equivalence class labelled $y$.  With these three eigenfunctions, we can form a nontrivial linear combination $F$ such that $Q(F,F)<0$ and $\int_{\partial\Omega_{i}}F\gamma_{\sdimn}(x)dx=0$ for all $i=1,2,3$, violating the minimality of $\Omega_{1},\Omega_{2},\Omega_{3}$.  Therefore,  $\Omega_{1},\Omega_{2},\Omega_{3}$ must have exactly three connected components.

In summary, $\Omega_{1},\Omega_{2},\Omega_{3}$ must have exactly three connected components and flat boundaries, so that either $\Omega_{1},\Omega_{2},\Omega_{3}$ are the sets described in Conjecture \ref{conj0}, or $\Omega_{1},\Omega_{2},\Omega_{3}$ consist of three parallel slabs.  Having reduced to only two cases, one can conclude by showing that the slabs do not minimize Problem \ref{prob1} for $m=3$ by either direct estimates, or the differentiation argument of \cite{milman18a}, which we recall below in the more general case $m=4$.
\end{proof}

\section{Proof of (Conditional) Triple Bubble Conjecture}\label{sectrip}

Below, we let $^{t}$ denote the transpose of a matrix or vector.  Also, all vectors are assumed to be column vectors.

\begin{proof}[Proof of Corollary \ref{cor2}]
Let $m=4$.  Existence of $\Omega_{1},\ldots,\Omega_{4}$ follows from Lemma \ref{lemma51p}.  Assumption \ref{as1} holds by Lemma \ref{lemma52.6}.  By Theorem \ref{thm0}, we may assume that $\Omega_{1},\ldots,\Omega_{4}\subset\R^{3}$.  From Theorem \ref{thm1}, either all connected components of $\redb\Omega_{1},\redb\Omega_{2},\redb\Omega_{3}$ are contained in flat planes, or we may assume that $\Omega_{1},\ldots,\Omega_{4}\subset\R^{2}$.  If $\Omega_{1},\ldots,\Omega_{4}\subset\R$, then a rearrangement argument implies that $\Omega_{1},\ldots,\Omega_{4}$ are connected intervals, so we ignore this case for now.

In the first case, we can choose three nonzero, linearly independent vectors in $\R^{3}$ each corresponding to an eigenfunction of $\L$ with eigenvalue $1$ by Remark \ref{rk30}.

We now consider the case that $\Omega_{1},\ldots,\Omega_{4}\subset\R^{2}$.  In this last remaining case, either $\pcon=1$ or $\pcon>1$, since $\pcon\geq1$ by Lemma \ref{lemma23}.  In the former case, we claim that for any $\epsilon>0$, there exists a three-dimensional subspace of functions such that $Q(G,G)\leq-(1-\epsilon)\langle G,G\rangle$ for every $G$ in this subspace.  To see this, first suppose $1<\pcon<\infty$.  Note that we have two eigenfunctions with eigenvalue $1$ from Remark \ref{rk30} and another Dirichlet eigenfunction $F$ of $\L$ since $\pcon>1$.  Consider the three-dimensional subspace formed by linear combinations of $F$ and $G$, where $G$ is an eigenfunction with eigenvalue $1$.  Let $a\colonequals\langle F,F\rangle^{1/2}$, $b\colonequals\langle G,G\rangle^{1/2}$.  We estimate $Q(F+G,F+G)$.  Without loss of generality, $a^{2}+b^{2}=1$.  Then
$$-Q(F+G,F+G)=\pcon a^{2}+(\pcon+1)\langle F,G\rangle+b^{2}=1+(\pcon-1)a^{2}+(\pcon+1)\langle F,G\rangle.$$
From Lemma \ref{rkorth}, for any $\epsilon>0$, we have $\abs{(\pcon+1)\langle F,G\rangle}\leq2(\pcon+1)(\pcon-1)^{-1/2}\epsilon ab$.  So,
$$-Q(F+G,F+G)\geq 1+(\pcon-1)a^{2}-2(\pcon+1)(\pcon-1)^{-1/2}\epsilon ab.$$
The function $a\mapsto(\pcon-1)a^{2}-2(\pcon+1)(\pcon-1)^{-1/2}\epsilon ab$ for $a\in[-1,1]$ has a minimum value of $-b^{2}\epsilon^{2}(\pcon+1)^{2}/(\pcon-1)^{2}$.  So, if $\epsilon$ is small enough, we have $-Q(F+G,F+G)\geq1-\epsilon$.  Meanwhile, $\langle F+G,F+G\rangle=1+2\langle F,G\rangle\leq1+2\epsilon(\pcon-1)^{-1/2}ab$.  The claim follows.  (This argument also applies to taking Dirichlet eigenfunctions in the case $\pcon=\infty$.)
%   a mapsto ca^2 +da.  derivative = 2ca+d=0. a=-d/2c
%  function is cd^2 /4c^2  -d^2 /2c = -d^2 /4c.  = (pcon+1)^2 /(pcon-1)^2

In summary, in any case, for any $\epsilon>0$ there is a three-dimensional subspace where $Q(G,G)\leq-(1-\epsilon)\langle G,G\rangle$ for all $G$ in the subspace, all such $G$ satisfy \eqref{four75}, and there are $G_{1},G_{2},G_{3}$ in this subspace such that
\begin{equation}\label{ten1}
\mathrm{span}\Big\{\Big(\int_{\redb\Omega_{1}}\sum_{j\neq 1}g_{k,1j}\gamma_{\sdimn}(x)dx,\ldots,\int_{\redb\Omega_{4}}\sum_{j\neq m}g_{k,mj}\gamma_{\sdimn}(x)dx\Big)\colon k=1,2,3\Big\}
\end{equation}
is three (which is as large as possible).  (Otherwise, there is a nonzero element $G$ of the subspace such that the vector in \eqref{ten1} is zero while $Q(G,G)<0$, violating the minimality of $\Omega_{1},\ldots,\Omega_{4}$.)

Let $m=4$.  Let $\Delta_{m}\colonequals\{a=(a_{1},\ldots,a_{m})\colon a_{1}+\cdots+a_{m}=1,\,a_{i}>0\,\,\forall\,\,1\leq i\leq m\}$.  Let $a\in\Delta_{m}$.  Let $z_{1},\ldots,z_{m}\in\R^{m-1}$ be the vertices of a regular simplex in $\R^{m-1}$ centered at the origin.  Assume that $\vnormt{z_{i}}=1$ for all $1\leq i\leq m$.  For all $1\leq i\leq m$, define
$$\Omega_{i}(a)'\colonequals y+\{x\in\R^{m-1}\colon\langle x,z_{i}\rangle=\max_{1\leq j\leq m}\langle x,z_{j}\rangle\},$$
where $y\in\R^{m-1}$ is chosen so that $\gamma_{m-1}(\Omega_{i}'(a))=a_{i}$ for all $1\leq i\leq m$.

For any $a\in\Delta_{m}$, define
\begin{flalign*}
J(a)
&\colonequals \min\{\int_{\cup_{i=1}^{m}\partial\Omega_{i}}\gamma_{m-2}(x)dx\colon \cup_{i=1}^{m}\Omega_{i}=\R^{m-1},
\,\gamma_{m-1}(\Omega_{i})=a_{i},\,\forall\,1\leq i\leq m\}\\
&\qquad-\int_{y+\cup_{i=1}^{m}\partial\Omega_{i}'(a)}\gamma_{m-2}(x)dx.
\end{flalign*}
By definition of $J$, we have $J(a)\leq0$.  Corollary \ref{cor1} is the assertion that $J(a)=0$ for all $a\in\Delta_{m}$.  We argue by contradiction.  Assume $J(a)<0$ for some $a\in\Delta_{m}$.  Further, assume that
$$J(a)=\min_{a'\in\Delta_{m}}J(a').$$
This minimum exists since $J(a')=0$ for any $a'$ with $a_{1}'+\cdots+a_{m}'=1$ and $a_{i}'=0$ for some $1\leq i\leq m$ by the known $m=3$ case of Conjecture \ref{conj0}, recalling that $m=4$ in this section.  For any $a\in\Delta_{m}$, define $I(a)$ as in Lemma \ref{lemma75}.

Let $(\Omega_{1}^{(s)},\ldots,\Omega_{m}^{(s)})_{s\in(-1,1)}$ be a smoothly varying partition of $\R^{\adimn}$ such that $(\Omega_{1}^{(0)},\ldots,\Omega_{m}^{(0)})$ minimizes Problem \ref{prob1} when $s=0$, and define $\Sigma_{ij}^{(s)}\colonequals(\redb\Omega_{i}^{(s)})\cap(\redb\Omega_{j}^{(s)})$ for any $1\leq i<j\leq m$ and for any $s\in(-1,1)$.  For any $s\in(-1,1)$, define
\begin{equation}\label{eleven0}
V(s)\colonequals\Big(\gamma_{\adimn}(\Omega_{1}^{(s)}),\ldots,\gamma_{\adimn}(\Omega_{m}^{(s)})\Big).
\end{equation}
Then $J(V(s))$ has a local minimum at $s=0$, so its first derivative is zero, and its second derivative is nonnegative, i.e.
\begin{equation}\label{eleven1}
\langle J'(V(0)),V'(0)\rangle =0=\frac{d}{ds}|_{s=0}\sum_{1\leq i<j\leq m}\int_{\redb\Sigma_{ij}^{(s)}}\gamma_{\sdimn}(x)dx-\langle I'(V(0)),V'(0)\rangle.
\end{equation}
\begin{equation}\label{eleven3}
\langle I'(V(0)),V''(0)\rangle +(V'(0))^{t}I''(V(0))V'(0)
\leq \frac{d^{2}}{ds^{2}}|_{s=0}\sum_{1\leq i<j\leq m}\int_{\Sigma_{ij}^{(s)}}\gamma_{\sdimn}(x)dx.
\end{equation}
From Remark \ref{rk27}, (using the notation from there),
\begin{equation}\label{eleven3.5}
\frac{d}{ds}|_{s=0}\sum_{1\leq i<j\leq m}\int_{\Sigma_{ij}^{(s)}}\gamma_{\sdimn}(x)dx=\sqrt{2\pi}\langle\lambda,V'(0)\rangle.
\end{equation}
By e.g. \eqref{ten1}, $V'(0)$ can be chosen to be any vector in $\Delta_{4}$, so \eqref{eleven3.5} and \eqref{eleven1} imply that $I'(V(0))=\sqrt{2\pi}\lambda$.  We can then rewrite \eqref{eleven3} as
\begin{equation}\label{eleven4}
(V'(0))^{t}I''(V(0))V'(0)
\leq \frac{d^{2}}{ds^{2}}|_{s=0}\sum_{1\leq i<j\leq m}\int_{\Sigma_{ij}^{(s)}}\gamma_{\sdimn}(x)dx-\sqrt{2\pi}\langle\lambda,V''(0)\rangle.
\end{equation}
Using Lemma \ref{lemma21}, the second part of Remark \ref{rk27} together with \eqref{quaddef}, and using any of the vector fields from \eqref{ten1} (thereby defining $F\in\mathcal{F}$ with $f_{ij}\colonequals\langle X,N_{ij}\rangle$ for any $1\leq i<j\leq m$, where $X$ satisfies \eqref{nine2.3} for $\Psi$ defined so that $\Omega_{i}^{(s)}=\Psi(\Omega_{i}^{(0)},s)$ for all $s\in(1,1)$), \eqref{eleven4} becomes
\begin{equation}\label{eleven5}
(V'(0))^{t}I''(V(0))V'(0)\leq Q(F,F).
\end{equation}
Note that \eqref{four75} holds for $F$, justifying \eqref{eleven5}.

We now apply Lemmas \ref{lemma86}, \ref{lemma87} and \ref{lemma75} to \eqref{eleven5} to get, $\forall$ $\epsilon>0$,
$$(V'(0))^{t}I''(V(0))V'(0)\leq-(1-\epsilon)(\sqrt{2\pi}V'(0))^{t}K^{-1}\sqrt{2\pi}V'(0).$$
By \eqref{ten1}, we then have $I''\leq -2\pi K^{-1}$ in the positive semidefinite sense of matrices acting on $\Delta_{m}$.  Since both matrices have the same null space, we conclude that $-I''\leq -2\pi K^{-1}$ in the positive semidefinite sense of $m\times m$ matrices.  But then
$$2I(V(0))\stackrel{\eqref{twelve1}}{=}-\frac{1}{2\pi}\mathrm{Tr}((I''(V(0)))^{-1})]\leq\mathrm{Tr}(K)\stackrel{\eqref{seven4}}{=}2\sum_{1\leq i<j\leq m}\gamma_{\sdimn}(\Sigma_{ij}).$$
That is, $J(a)\geq0$, a contradiction.
\end{proof}

\section{Lemmas from Milman-Neeman}\label{secmilman}

Below, we let $^{t}$ denote the transpose of a matrix or vector.  Also, all vectors are assumed to be column vectors.

\begin{lemma}[{\cite[Lemma 6.4]{milman18a}}]\label{lemma86}
Let $D,E$ be random vectors in $\R^{\adimn}$ such that $\E\vnorm{D}^{2}<\infty$ and $\E\vnorm{E}^{2}<\infty$.  Assume that $\E EE^{t}$ is nonsingular.  Then
$$(\E DE^{t})(\E EE^{t})^{-1}(\E ED^{t})\leq \E DD^{t}.$$
\end{lemma}

Let $u_{1},\ldots,u_{\adimn}$ denote the standard basis of $\R^{\adimn}$.  Let $\Omega_{1},\ldots,\Omega_{m}$ minimize Problem \ref{prob1}.  Define
\begin{equation}\label{seven4}
K\colonequals\sum_{1\leq i<j\leq m}\gamma_{\sdimn}(\Sigma_{ij})(u_{i}-u_{j})(u_{i}-u_{j})^{t}.
\end{equation}

\begin{lemma}\label{lemma87}
Let $\pcon\geq1$.  Let $F\in\mathcal{F}$ satisfy $Q(F,F)\leq-\pcon\langle F,F\rangle$ and \eqref{four75}.  Define $V$ by \eqref{eleven0}.  Then
$$Q(F,F)\leq -\pcon (\sqrt{2\pi}V'(0))^{t}K^{-1}(\sqrt{2\pi}V'(0)).$$
\end{lemma}
\begin{proof}
For any $1\leq i<j\leq m$, define $\E_{ij}f_{ij}\colonequals\int_{\Sigma_{ij}}f_{ij}\gamma_{\sdimn}(x)dx/\int_{\Sigma_{ij}}\gamma_{\sdimn}(x)dx$ and define
$$\overline{N}_{ij}\colonequals \E_{ij}(N_{ij}|X),\qquad S\colonequals\sum_{1\leq i<j\leq m}\overline{N}_{ij}\overline{N}_{ij}^{t}.$$
Define $D,E$ so that $D\colonequals \overline{N}_{ij}$ and $E\colonequals u_{i}-u_{j}$ on $\Sigma_{ij}$, for all $1\leq i<j\leq m$.  Define
\begin{equation}\label{seven4p}
M\colonequals\sum_{1\leq i<j\leq m}\gamma_{\sdimn}(\Sigma_{ij})(u_{i}-u_{j})\overline{N}_{ij}^{t}.
\end{equation}
Then by Definition \ref{quaddef} and \eqref{four75},
\begin{flalign*}
Q(F,F)
&\leq-\pcon\sum_{1\leq i<j\leq m}\int_{\Sigma_{ij}}f_{ij}^{2}\gamma_{\sdimn}(x)dx
=-\pcon\sum_{1\leq i<j\leq m}\frac{\int_{\Sigma_{ij}}\langle X,N_{ij}\rangle^{2}\gamma_{\sdimn}(x)dx}{\gamma_{\sdimn}(\Sigma_{ij})}\gamma_{\sdimn}(\Sigma_{ij})\\
&=-\pcon\sum_{1\leq i<j\leq m}\E_{ij}\langle X,N_{ij}\rangle^{2}\gamma_{\sdimn}(\Sigma_{ij})
=-\pcon\sum_{1\leq i<j\leq m}\E_{ij}\E_{ij}(\langle X,N_{ij}\rangle^{2}|X)\gamma_{\sdimn}(\Sigma_{ij})\\
\end{flalign*}
So, using the conditional Jensen inequality,
\begin{equation}\label{seven7}
\begin{aligned}
&Q(F,F)\\
&\leq-\pcon\sum_{1\leq i<j\leq m}\E_{ij}[\E_{ij}(\langle X,N_{ij}\rangle|X)]^{2}\gamma_{\sdimn}(\Sigma_{ij})
=-\pcon\sum_{1\leq i<j\leq m}\E_{ij}[\langle X,\E_{ij}(N_{ij}|X)\rangle]^{2}\gamma_{\sdimn}(\Sigma_{ij})\\
&=-\pcon\sum_{1\leq i<j\leq m}\E_{ij}X^{t}\overline{N}_{ij}\overline{N}_{ij}^{t}X\gamma_{\sdimn}(\Sigma_{ij})
=-\pcon\E X^{t}DD^{t}X\gamma_{\sdimn}(\Sigma).
\end{aligned}
\end{equation}
For any $g\colon\Sigma\to\R$, we defined $\E g\colonequals\int_{\Sigma}g(x)\gamma_{\sdimn}(x)dx/\gamma_{\sdimn}(\Sigma)$.  Note that
$$ \E (ED^{t})=\E (ED^{t}|X)=\frac{1}{\gamma_{\sdimn}(\Sigma)}\E(M|X),\qquad
\E (EE^{t})=\E (EE^{t}|X)=\frac{1}{\gamma_{\sdimn}(\Sigma)}K.$$
So, from Lemma \ref{lemma86},
$$-\E (DD^{t}|X)\leq-(\E (M|X))^{t}K^{-1}(\E (M|X))\gamma_{\sdimn}(\Sigma).$$
Multiplying by $X^{t}$ on the left and $X$ on the right, then taking $\E$ of both sides,
$$-\E (X^{t}DD^{t}X)\leq-(\E MX)^{t}K^{-1}(\E MX)\gamma_{\sdimn}(\Sigma).$$
Consequently,
\begin{equation}\label{seven8}
Q(F,F)\stackrel{\eqref{seven7}}{\leq}
 -\pcon (\E MX)^{t}K^{-1}(\E MX)\gamma_{\sdimn}(\Sigma)^{2}.
\end{equation}
Note now that
\begin{flalign*}
&\gamma_{\sdimn}(\Sigma)\E MX\\
&=\sum_{1\leq i<j\leq m}\gamma_{\sdimn}(\Sigma_{ij})(u_{i}-u_{j})\E_{ij} \overline{N}_{ij}^{t}X
=\sum_{1\leq i<j\leq m}\gamma_{\sdimn}(\Sigma_{ij})(u_{i}-u_{j})\E_{ij} \langle\E_{ij}(N_{ij}\rangle|X),X\rangle\\
&=\sum_{1\leq i<j\leq m}\gamma_{\sdimn}(\Sigma_{ij})(u_{i}-u_{j})\E_{ij} \E_{ij}(\langle N_{ij},X\rangle|X)
=\sum_{1\leq i<j\leq m}\gamma_{\sdimn}(\Sigma_{ij})(u_{i}-u_{j})\E_{ij} \langle N_{ij},X\rangle.
\end{flalign*}
Now use Remark \ref{rk23} to see that $\sqrt{2\pi}V'(0)=\gamma_{\sdimn}(\Sigma)\E MX$, so \eqref{seven8} concludes the proof.
\end{proof}

\begin{lemma}[\embolden{Differentiation Formula}, {\cite[Proposition 2.6]{milman18a}}]\label{lemma75}
Let $z_{1},\ldots,z_{m}\in\R^{\adimn}$ be the vertices of a regular simplex in $\R^{\adimn}$ centered at the origin.  For all $1\leq i\leq m$, define
$$\widetilde{\Omega}_{i}\colonequals \{x\in\R^{\adimn}\colon\langle x,z_{i}\rangle=\max_{1\leq j\leq m}\langle x,z_{j}\rangle\}.$$
For any $y\in\R^{\adimn}$, define
$$B(y)\colonequals\int_{y+\cup_{i=1}^{m}\partial\widetilde{\Omega}_{i}}\gamma_{\sdimn}(x)dx.$$
$y=y(a)$ such that $\gamma_{\adimn}(\widetilde{\Omega_{i}})=a_{i}>0$, $\forall$ $1\leq i\leq m$, $\forall$ $a=(a_{1},\ldots,a_{m})\in\R^{m}$ with $\sum_{i=1}^{m}a_{i}=1$, and define $I(a)\colonequals B(y(a))$.  Then for all $b=(b_{1},\ldots,b_{m})\in\R^{m}$ with $\sum_{i=1}^{m}b_{i}=0$, and letting $\overline{\nabla},\overline{\Delta}$ denote the gradient and Laplacian on $\R^{\adimn}$, respectively,
$$\overline{\nabla}_{b} I(a)=\sqrt{\pi}(a_{1}^{-1},\ldots,a_{m}^{-1})b$$
\begin{equation}\label{twelve1}
\overline{\nabla}_{b}\overline{\nabla}_{b} I(a)=-2\pi b\widetilde{K}^{-1}b.
\end{equation}
Here $\widetilde{K}\colonequals\sum_{1\leq i<j\leq m}\gamma_{\sdimn}((\partial\widetilde{\Omega_{i}})\cap(\partial\widetilde{\Omega_{j}}))(u_{i}-u_{j})(u_{i}-u_{j})^{t}$.
\end{lemma}

\section{Comments on More than Four Sets}\label{secconc}

It would be desirable to remove our need for Assumption \ref{as2} in the cases $m=3,4$ of Conjecture \ref{conj0}.  Moreover, it would be nice to extend our arguments to the case $m>4$ of Conjecture \ref{conj0}.  If $\Omega_{1},\ldots,\Omega_{m}$ minimize Problem \ref{prob1} and the boundaries of the sets are not all flat, it follows from Remark \ref{rk30} and the reasoning of Lemma \ref{lemma33} that there is a subspace of eigenfunctions of the second variation operator $\L$ from \eqref{seven00} of dimension $m-1+\adimn=\sdimn+m$ with positive eigenvalues.  As discussed in the introduction, this fact alone should be sufficient to solve Conjecture \ref{conj0}, since more eigenfunctions of $\L$ means more control on the structure of $\Omega_{1},\ldots,\Omega_{m}$.  However, it could be the case that the quadratic form $Q$ from \eqref{quaddef} is negative definition on most of these $\sdimn+m$ eigenfunctions, contrary to our intuition.  (Recall that $Q$ came from Lemmas \ref{lemma28} and \ref{lemma29.9}.)  The main issue is that we have no a priori control on the second term in \eqref{quaddef} for these eigenfunctions of $\L$.  One might think that a Sobolev trace inequality could control this second term of \eqref{quaddef}, but then it seems we would need e.g. a bounded geometry condition on $\Omega_{1},\ldots,\Omega_{m}$ to proceed further.

\medskip
\noindent\textbf{Acknowledgement}.  Thanks to David Galvin, Russ Lyons, Yury Makarychev, Emanuel Milman and Elchanan Mossel for helpful discussions.

\bibliographystyle{amsalpha}
\bibliography{12162011}

\end{document}